\documentclass[11pt,regno]{amsart}
\usepackage{euscript}
\usepackage[pdftex]{graphicx,color}
\newtheorem{theorem}{Theorem}
\newtheorem{lemma}{Lemma}
\newtheorem{corollary}{Corollary}
\newtheorem{proposition}{Proposition}
\theoremstyle{definition}

\newtheorem{remark}{Remark}
\newtheorem*{remark*}{Remark}

\newcommand{\eqdef}{\stackrel{\scriptscriptstyle\rm def}{=}}

\DeclareMathOperator{\diam}{diam}

\DeclareMathOperator{\dist}{Dist}

\DeclareMathOperator{\Crit}{{\rm Crit}}

\def\bN{\mathbb{N}}\def\bC{\mathbb{C}}

\def\bR{\mathbb{R}}

\def\cU{\EuScript{U}}

\def\cM{\EuScript{M}}

\def\cL{\EuScript{L}}

\def\ccL{\widehat\cL}

\DeclareMathSymbol{\varnothing}{\mathord}{AMSb}{"3F}
\renewcommand{\emptyset}{\varnothing}

\author{Katrin Gelfert} \address{IMPA, Estrada Dona Castorina 110, Rio de Janeiro, Brazil
}
\email{gelfert@pks.mpg.de}

\author{Feliks Przytycki} \address{Institute of Mathematics, Polish Academy of Sciences,
ul. \'{S}niadeckich 8, 00-956 Warszawa, Poland}
\email{feliksp@impan.gov.pl}
\urladdr{http://www.impan.gov.pl/~feliksp}

\author{{Micha\l} Rams} \address{Institute of Mathematics, Polish Academy of Sciences,
ul. \'{S}niadeckich 8, 00-956 Warszawa, Poland}
\email{rams@impan.gov.pl}
\urladdr{http://www.impan.gov.pl/~rams}

\begin{document}

\title[On the Lyapunov spectrum for rational maps]{Lyapunov spectrum for rational maps}

\begin{abstract}
We study the dimension spectrum of Lyapunov exponents for rational
maps on the Riemann sphere.
\end{abstract}

\begin{thanks}
{This research was supported by the EU FP6 Marie Curie programmes
SPADE2 and CODY. The research of F.\,P and M.\,R. were supported
by the Polish MNiSW Grant NN201 0222 33 `Chaos, fraktale i
dynamika konforemna'. K.\,G. was partially supported by the
Deutsche Forschungsgemeinschaft and by the Humboldt foundation. K.\,G. and M.\,R. thank
for the
hospitality of MPI~PKS Dresden and IM~PAN Warsaw where part of this
research was done.}
\end{thanks}


\maketitle
\tableofcontents

\section{Introduction}

Let $f\colon \overline\bC\to\overline\bC$ be a rational function of degree
$d\ge 2$ on the Riemann sphere and let $J=J(f)$ be its Julia
set. 
Our goal is to study the spectrum of Lyapunov exponents of $f|_J$.
Given $x\in J$ we denote by $\underline\chi(x)$ and
$\overline\chi(x)$ the \emph{lower} and \emph{upper Lyapunov
exponent} at $x$, respectively, where
\[
\underline\chi(x)\eqdef\liminf_{n\to\infty}\frac{1}{n}\log\, \lvert(f^n)'(x)\rvert,\quad
\overline\chi(x)\eqdef\limsup_{n\to\infty}\frac{1}{n}\log\, \lvert(f^n)'(x)\rvert.
\]
If both values coincide then we call the common value the
\emph{Lyapunov exponent} at $x$ and denote it by $\chi(x)$. For
given numbers $0\le \alpha\le\beta$ we consider also the following
level sets
\[
\cL(\alpha,\beta) \eqdef
 \left\{x\in J\colon \underline\chi(x)=\alpha, \, \overline\chi(x)=\beta\right\}.
\]
We denote by $\cL(\alpha)\eqdef
\cL(\alpha,\alpha)$ the set of \emph{Lyapunov regular points} with  exponent
$\alpha$.
If $\alpha<\beta$ then $\cL(\alpha,\beta)$ is contained in the set of so-called
\emph{irregular points}
\[
\cL_{\rm irr}\eqdef\left\{ x\in J\colon \underline\chi(x)<\overline\chi(x)\right\}.
\]
Recall that it follows from the Birkhoff ergodic theorem that
$\mu(\cL_{\rm irr})=0$ for any $f$-invariant probability measure
$\mu$.

While the first results on the multifractal formalism go already back
to Besicovitch~\cite{Bes:34}, its systematic study has been
initiated by work of Collet, Lebowitz and
Porzio~\cite{ColLebPor:87}. The case of spectra of Lyapunov
exponents for conformal uniformly expanding repellers has been
covered for the first time in~\cite{BarPesSch:97} building also on
work by Weiss~\cite{Wei:99} (see~\cite{Pes:97} for more details
and references).
 To our best knowledge, the first results on
irregular parts of a spectrum were obtained in~\cite{Bes:34}.
Its first complete description  (for digit expansions) was given
by Barreira, Saussol, and Schmeling~\cite{BarSauSch:02}.

In this work we will formulate our results on  the spectrum of
Lyapunov exponents in terms of the topological pressure $P$. For
any $\alpha> 0$ let us first define
\[
F(\alpha)\eqdef 
\frac{1}{\alpha}\inf_{d\in\bR}\left( P_{f|J}(-d\log\,\lvert f'\rvert)+d\alpha\right).
\]
and
\[
F(0)\eqdef \lim_{\alpha\to 0+}F(\alpha).
\]
Let $[\alpha^-,\alpha^+]$ be the interval on which $F\neq -\infty$
(a more formal and equivalent definition is given in~\eqref{yu} below).

 Before stating our main results, we recall what is already known in the case
that $J$ is a uniformly expanding repeller with respect to $f$,
that is, $J$ is a compact $f$-forward invariant (i.e., $f(J)=J$)
isolated set such that $f|_J$ is uniformly expanding. Recall that
$f$ is said to be \emph{uniformly expanding} or \emph{uniformly
hyperbolic} on a set $\Lambda$ if there exist $c>0$ and $\lambda>1$ such
that for every $n\ge 1$ and every $x\in \Lambda$ we have $\lvert
(f^n)'\rvert\ge c \lambda^n$. Recall that a set $\Lambda$ is said to be
\emph{isolated} if there exists an open neighborhood $U\subset
\overline\bC$ of $\Lambda$ such that $f^n(x)\in U$ for every $n\ge 0$
implies $x\in \Lambda$.
In our setting the Julia set $J$ is  a uniformly expanding
repeller if it does not contain any critical point nor parabolic
point. Here a point $x$ is said to be \emph{critical} if $f'(x)=0$
and to be \emph{parabolic} if $x$ is periodic and its multiplier $
(f^{{\rm per}(x)})'(x)$ is a root of $1$. If $J$ is a uniformly
expanding repeller then for every $\alpha\in[\alpha^-,\alpha^+]$
we have~
\[
\dim_{\rm H}\cL(\alpha) =F(\alpha)
\]
(see~\cite{BarPesSch:97,Wei:99}) and $\cL(\alpha)=\emptyset$ if
and only if $\alpha\notin[\alpha^-,\alpha^+]$~\cite{Sch:97}. This
gives the full description of the regular part of the Lyapunov
spectrum. Moreover, in this setting of a uniformly expanding
repeller $J$, the interval $[\alpha^-,\alpha^+]$ coincides with
the closure of the range of the function $\alpha(d) =
-\frac{d}{ds}P_{f|J}(-s\log\,\lvert f'\rvert)|_{s=d}$ and the
spectrum can be written as

\[
F(\alpha(d)) = \frac{1}{\alpha(d)}\left( P_{f|J}(-d\log\,\lvert
f'\rvert)+d\alpha(d)\right) =\frac{h_{\mu_d}(f)}{\alpha(d)},
\]
where $\alpha(d)$ is the unique number satisfying
\[
-\frac{d}{ds}P_{f|J}(-s\log\,\lvert f'\rvert)|_{s=d} = \int\log\,\lvert f'\rvert\,d\mu_d
= \alpha(d)
\]
and where $\mu_d$ is the unique equilibrium state corresponding to
the potential $-d\log\,\lvert f'\rvert$. If $\log\,\lvert
f'\rvert$ is not cohomologous to a constant then we have
$\alpha^-<\alpha^+$, and $\alpha\mapsto \alpha F(\alpha)$ and
$d\mapsto P_{f|J}(-d\log\,\lvert f'\rvert)$ are real analytic
strictly convex functions that form a Legendre pair.

We now state our first main result.

\begin{theorem} \label{main1}
Let $f$ be a rational function of degree $\ge 2$ with
no critical points in its Julia set $J$. For
any $\alpha^-\le \alpha\le \beta \le \alpha^+$, $\beta>0$, we have
\[
\dim_{\rm H} \cL(\alpha, \beta)= \min\{F(\alpha), F(\beta)\}\,.
\]
In particular, for any $\alpha\in [\alpha^-,\alpha^+]\setminus\{0\}$ we have
\[
\dim_{\rm H} \cL(\alpha) = F(\alpha)\,.
\]
If there exists a parabolic point in $J$ (and hence $F(0)>-\infty$) then
\[
\dim_{\rm H} \cL(0) =\dim_{\rm H} J
= F(0)\,.
\]
Moreover,
\[
\left\{x\in J\colon \underline{\chi}(x)<\alpha^-\right\}
= \left\{x\in J\colon \overline{\chi}(x)>\alpha^+ \right\} = \emptyset.
\]
\end{theorem}

We denote by $\Crit$ the set of all critical points of $f$.
Following Makarov and Smirnov~\cite[Section 1.3]{MakSmi:00}, we
will say that $f$ is \emph{exceptional} if there exists a finite,
nonempty set $\Sigma_f\subset \overline\bC$ such that
\[
f^{-1}(\Sigma_f)\setminus {\rm Crit} = \Sigma_{f} \,.
\]
This set need not be unique. We will further denote by $\Sigma$
the largest of such sets (notice that it has no more than 4 points).

\begin{theorem} \label{main2}
Let $f$ be a rational function of degree $\ge 2$. Assume that $f$ is
non-exceptional or that $f$ is exceptional but $\Sigma\cap
J=\emptyset$. For any $0< \alpha\le \beta \le \alpha^+$ we have
\[
\min\{F(\alpha), F(\beta)\}\le \dim_{\rm H} \cL(\alpha, \beta)\le
\max_{\alpha\leq q \leq \beta}F(q)\,.
\]
In particular, for any $\alpha\in [\alpha^-,\alpha^+] \setminus \{0\}$ we have
\[
\dim_{\rm H} \cL(\alpha) = F(\alpha)
\]
and
\[
\dim_{\rm H} \cL(0) \geq F(0)\,.
\]
Moreover,
\[
\left\{ x\in J\colon -\infty<\chi(x)<\alpha^-\right\}
 = \left\{ x\in J\colon \overline{\chi}(x)>\alpha^+\right\}
 = \emptyset
\]
and
\[
\dim_{\rm H}\left\{ x\in J\colon 0<\overline{\chi}(x)<\alpha^-\right\} = 0.
\]
\end{theorem}

If $f$ is exceptional and $\Sigma\cap J\ne \emptyset$ (this happens, for
example, for Chebyshev polynomials and some Latt\`es maps) then the
situation can be much different from the above-mentioned cases.
For example, the map $f(x)=x^2-2$ possesses countably many points
with Lyapunov exponent $-\infty$, two points with Lyapunov
exponent $2\log 2$, a set of dimension 1 of points with Lyapunov
exponent $\log 2$, and no other Lyapunov regular points. Hence,
for this map the Lyapunov spectrum is not complete in the interval
$[\alpha^-, \alpha^+]=[\log 2, 2\log 2]$.

The present paper does not provide a complete description of the
irregular part of the Lyapunov spectrum even in the case
$\Sigma\cap J=\emptyset$. We do not know how big the set $\cL(-\infty)$ is
except in the case when $f$ has only one critical point in $J$ (in
which case $\cL(-\infty)$ consists only of the backward orbit of
this critical point). Moreover, we do not know whether the set
$\{x\in J\colon \underline{\chi}(x)<\alpha^-\}$ contains any
points other than the backward orbits of critical points contained in
$J$ and we only have some estimation for the Hausdorff dimension
of the set $\cL(\alpha,\beta)$ even for values $\alpha$, $\beta\in
[\alpha^-,\alpha^+]$. 

The paper is organized as follows.
In Section~\ref{sec:tools} we introduce the tools we are going to use in this paper. In
particular, we construct a family of uniformly expanding Cantor repellers
with pressures pointwise converging to the pressure on $J$ (Proposition~\ref{decaf}).
Section~\ref{sec:3.5} discusses general properties of the spectrum of exponents.
In Section~\ref{sec:up} we obtain upper bounds for the Hausdorff dimension. Here we
use conformal measures to deal with conical points (Proposition~\ref{prop:1}) and we
prove that the set of non-conical points with positive upper Lyapunov exponent is very
small using the pullback construction (Proposition~\ref{prop:wander}).
In Section~\ref{sec:4} we derive lower bounds for the dimension. To do so, we first
consider the interior of the spectrum and we will use the  sequence of Cantor
repellers from Section~\ref{sec:tools} to obtain for any $\alpha\in (\alpha^-,\alpha^+)$
a big uniformly expanding subset of points with Lyapunov exponent $\alpha$ from which we
derive our estimates. We finally study the boundary of the spectrum and the irregular
part of the spectrum using a construction, that generalizes the w-measure construction
from~\cite{GelRam:}.

\newpage
\section{Tools for non-uniformly hyperbolic dynamical systems}\label{sec:tools}

\subsection{Topological pressure}\label{sec:press}

Given a compact $f$-invariant set $\Lambda\subset J$, we denote by $\cM(\Lambda)$ the
family of $f$-invariant Borel probability measures supported on $\Lambda$. We denote by
$\cM_{\rm E}(\Lambda)$ the subset of ergodic measures. Given $\mu\in\cM_{\rm
E}(\Lambda)$, we
denote by
\[
\chi(\mu)\eqdef\int_\Lambda\log\,\lvert f'\rvert\,d\mu
\]
the Lyapunov exponent of $\mu$. Notice that we have $\chi(\mu)\ge
0$ for any $\mu\in\cM(J)$~\cite{Prz:93}.

 Given $d\in\bR$, we define
the function $\varphi_d\colon J\setminus \Crit\to\bR$ by
\begin{equation}\label{varphit}
        \varphi_d(x) \eqdef
         -d\log\,\lvert f'(x)\rvert.
\end{equation}
Given a compact $f$-invariant uniformly expanding set
$\Lambda\subset J$,  the \emph{topological pressure} of
$\varphi_d$ (with respect to $f|_\Lambda$) is defined by
\begin{equation}\label{wirm}
 P_{f|\Lambda}(\varphi_d)\eqdef
      \max_{\mu\in\cM(\Lambda)}\left( h_\mu(f)+\int_\Lambda \varphi_d \,d\mu\right),
\end{equation}
where $h_\mu(f)$ denotes the entropy of $f$ with respect to $\mu$.
We simply write $\cM=\cM(J)$ and $P(\varphi_d)=P_{f|J}(\varphi_d)$
if we consider the full Julia set $J$ and if there is no confusion
about the system. A measure $\mu\in\cM$ is called
\emph{equilibrium state} for the potential $\varphi_d$ (with
respect to $f|_J$) if
\[
P(\varphi_d)=h_\mu(f)+\int_J\varphi_d\,d\mu.
\]

For every $d\in\bR$ we have the following equivalent
characterizations of the pressure function
(see~\cite{PrzRivSmi:04}, where further equivalences are shown).
We have
\begin{equation}
\begin{split}
P(\varphi_d) &=  \sup_{\mu\in\cM_{\rm E}^+}\left( h_\mu(f)+\int_J
\varphi_d \,d\mu\right)\\
&= \sup_{\Lambda}P_{f|\Lambda}(\varphi_d).
\end{split}\label{Phyp}
\end{equation}
Here in the first equality the supremum is taken over the set
$\cM_{\rm E}^+$ of all ergodic $f$-invariant Borel probability
measures on $J$ that have a positive Lyapunov exponent and are
supported on some $f$-invariant uniformly expanding subset of $J$.
In the second equality  the supremum is taken over all uniformly
expanding repellers $\Lambda\subset J$. In fact, in the second
equality it suffices to take the supremum over all uniformly
expanding  Cantor  repellers, that is, uniformly expanding
repellers  that are limit sets of finite graph directed systems
satisfying the strong separation condition with respect to $f$,
see Section~\ref{sec:hypsub}.

Let us introduce some further notation. Let
 \begin{equation}\label{yu}
 \begin{split}
\alpha^- &\eqdef
 \lim_{d\to\infty} - \frac 1 d P(\varphi_d) = \inf_{\mu\in \cM} \chi(\mu), \\
\alpha^+ &\eqdef
  \lim_{d\to -\infty} - \frac 1 d P(\varphi_d) = \sup_{\mu\in \cM} \chi(\mu),
\end{split}\end{equation}
where the given characterizations follow easily from the
variational principle.

Recall that, given $\alpha>0$, we define
\begin{equation}\label{def:Fa}
F(\alpha) \eqdef \frac{1}{\alpha} \inf_{d\in\bR}
\left(P(\varphi_d)+\alpha d \right)
\end{equation}
and
\begin{equation}\label{def:Fb}
F(0) \eqdef \lim_{\alpha\to 0+} F(\alpha).
\end{equation}
Note that
\[
F(0)=d_0\eqdef \inf\{d\colon P(\varphi_d )=0\}.
\]

\subsection{Building bridges between unstable islands}\label{sec:hypsub}

We describe a construction of connecting two given hyperbolic
subsets of the Julia set by ``building bridges" between the sets.%
\footnote{This is a precise realization of an idea of
Prado~\cite{Pra:}.}

We call a point $x\subset J$ \emph{non-immediately post-critical}
if there exists some preimage branch $x_0=x=f(x_1)$, $x_1=f(x_2)$,
$\ldots$ that is dense in $J$ and disjoint from $\Crit$. If $f$ is
non-exceptional or if it is exceptional but $\Sigma\cap
J=\emptyset$ then for every hyperbolic set all except possibly
finitely many points (in particular, all periodic points) are
non-immediately post-critical.


We will now consider  a set $\Lambda$ that is an $f$-uniformly
expanding Cantor repeller (ECR for short), that is,  a uniformly
expanding repeller and a limit set of a finite graph directed
system (GDS)  satisfying the strong separation condition (SSC)
with respect to $f$. Recall that by a GDS satisfying the SSC with
respect to $f$ we mean a family of domains and maps satisfying the
following conditions (compare~\cite[pp.~3,~58]{MauUrb:00}):
\begin{itemize}
\item[(i)] There exists a finite family $\cU= \{U_k\colon
k=1,...,K\}$ of open connected  (not necessarily simply connected)
domains in the Riemann sphere with pairwise disjoint closures.
\item[(ii)] There exists a family $G=\{g_{k\ell}\colon k,\ell\in
\{1,...,K\}\}$  of branches of $f^{-1}$ mapping $\overline{
U_\ell}$ into $U_k$ with bounded distortion (not all pairs
$k,\ell$ must appear here).\\
Note that a general definition of GDS allows many maps $g$ from
each $\overline{ U_\ell}$ to each $U_k$. Here however there can be
at most one, since we assume that $f$-critical points are far away
from $\Lambda$ and that the maps $g$ are branches of $f^{-1}$.
\item[(iii)] We have
\[
\Lambda= \bigcap_{n=1}^\infty\bigcup_{k_1,...,k_n} g_{k_1k_2}\circ
g_{k_2k_3}\circ ... \circ g_{k_{n-1}k_n}.
\]
We assume that we have $f(\Lambda)=\Lambda$ and hence that
for each $k$ there exists $\ell$ and for each $\ell$ there exists $k$ such that
$g_{k\ell}\in G$.
\end{itemize}
We can view $k=1$, $...$, $K$ as vertices and  $g_{k\ell}$ as
edges from $\ell$ to $k$ of a directed graph
$\Gamma=\Gamma({\cU},G)$.

This definition easily implies that $f$ is uniformly expanding on
the limit set $\Lambda$ of such a GDS, and that $\Lambda$ is a
repeller for $f$. Clearly $\Lambda$ is a Cantor set. In fact a
sort of converse is true (though we shall not use this fact in this
paper, but it clears up the definitions). Namely we observe  the
following fact.

\begin{lemma} \label{lem:1111}
    If $\Lambda\subset J$ is an $f$-invariant compact uniformly
expanding set that is  a Cantor set, then $\Lambda$ is contained
in the limit set of a GDS satisfying the SSC (this limit set can
be chosen to be contained in an arbitrarily small neighborhood of
$\Lambda$). Hence, $\Lambda$ is contained in an $f$-ECR set.
\end{lemma}

\begin{proof}
We can multiply the standard sphere Riemann metric by a positive
smooth function such that with respect to this new metric
$\rho_\Lambda$ we have $\lvert f'\rvert\ge\lambda>1$ on $\Lambda$.
It is easy to show that one can find an arbitrarily small number
$r>0$ such that  the neighborhood
$B(\Lambda,r)=\{z\in \overline \bC\colon \rho_\Lambda
(z,\Lambda)<r\}$ consists of a finite number of connected open
domains $U_k$ with pairwise disjoint closures. We account for our
GDS the branches of $g=f^{-1}$ on the sets $U_k$ such that each
$g(U_k)$ intersects $\Lambda$. Then $g$ maps each $\overline{U_k}$ into
some $U_\ell$ because it is a contraction (by the factor
$\lambda^{-1}$). Hence the family of maps $g|_{U_k}$ satisfies the assumptions of a GDS with  the SSC.
\end{proof}

In the proof of the following lemma we will ``build bridges'' between two ECR's.

\begin{lemma} \label{bridge}
For any  two disjoint $f$-ECR sets $\Lambda_1$, $\Lambda_2 \subset J$  that
both contain non-immediately postcritical points there exists an
$f$-ECR set $\Lambda\subset J$ containing the set $\Lambda_1\cup \Lambda_2$. If $f$
is topologically transitive on each $\Lambda_i$, $i=1$, $2$, then $f$ is
topologically transitive on $\Lambda$.
\end{lemma}

\begin{proof}
Let $\Lambda_1$, $\Lambda_2 \subset J$ be two sets satisfying the assumptions
of the lemma and let $p_1\in \Lambda_1$ and $p_2\in \Lambda_2$ be two
non-immediately postcritical points.
Consider a family ${\cU}_i=\{U_{i,k}\}_{k=1}^{K_i}$ of open
connected domains and a family
$G_i=\{g_{i,k\ell}\}$ of branches of  $f^{-1}$ mapping
$U_{i,\ell}$ to $U_{i,k}$ that define the DGS's satisfying the SSC
that have $\Lambda_i$ as their limit sets, for $i=1$, $2$, respectively.
Let
\[
D_i\eqdef\bigcup_{k=1}^{K_i} U_{i,k}.
\]
We can assume that each $D_i$ is an arbitrarily small neighborhood
of $\Lambda_i$, by replacing ${\cU}_i$ by $G^m(\cU_i)$, where by $G^m$
we denote the family of all compositions
\[
G^m=\left\{g_{i,k_1k_2}\circ g_{i,k_2k_3}\circ ... \circ g_{i,k_{m-1}k_m}\colon
g_{i,k_nk_{n+1}}\in G_i, n=1,\ldots,m-1\right\}.
\]
For each $i=1$, $2$ let us choose a backward trajectory $y_{i,t}$ of the point $p_i$
(the ``bridge'') such that 
\[
y_{i,0}=p_i, \quad
f(y_{i,t})=y_{i,t-1} \text{ for  }t=1,...,t_i,
\]
$y_{i,t}\notin \overline{ D_1 }\cup \overline{ D_2}$ for
all $t=1$, $\ldots$, $t_i-1$ and  $y_{1,t_1}\in D_2$, $y_{2,t_2}\in D_1$.
Let us denote by $h_{i,t}$ the branch of $f^{-t}$ that maps $p_i$ to  $y_{i,t}$, that is,
let
\[
h_{i,t} \eqdef f^{-t}_{y_{i,t}}.
\]
Let $V_i$ be an open disc centered at $p_i$ that is contained in $D_i$ and satisfies
$h_{i,t}(\overline{V_i})\cap (\overline{D_1}\cup\overline{D_2})=\emptyset$ for all $t=1$,
$\ldots$, $t_i-1$ (note that this
is possible provided we choose $V_i$ small enough) and
\[
h_{1,t_1}(\overline{V_1})\subset D_2
\text{ and }
h_{2,t_2}(\overline{V_2})\subset D_1.
\]

Let us consider an integer $N\ge 0$ such that the component of $f^{-N}(D_i)$ containing
$p_i$ is contained in $V_i$, that is, that
\[
f^{-N}_{p_i}(D_i)\subset V_i
\]
for $i=1$, $2$. Now let us replace ${\cU}_i$ by
$\widehat{\cU}_i\eqdef G^N({\cU}_i)$, let us replace each map
$g_{i,k\ell}$ by the family of its restrictions to $\widehat U\in
\widehat{\cU}_i$ contained in $U_\ell$ and and let us denote by
$\widehat G_i$ the union of those families. This defines a GDS
with graph $\Gamma_i=\Gamma_i(\widehat{\cU}_i, \widehat G_i)$, for
$i=1$, $2$. Now we restrict each bridge $h_{i,t}$ to the
element $\widehat V_i$ of $\widehat {\cU}_i$ that contains $p_i$.
As the next step we consider
\[
\widehat V_{i,t}\eqdef h_{i,t}(\widehat V_i) \quad\text{ for } t=1, ..., t_i+N-1,
\] where
for $t>t_i$ we choose an arbitrary prolongation of the bridge $y_{i,t}$
by maps $g_{j,k_tk_{t+1}}$. Finally, we consecutively thicken slightly $\widehat V_{i,t}$
along the bridges such that $f(\widehat V_{i,t})\supset \overline{\widehat V_{i,t-1}}$.
For each $t$ let us denote by $g_{i,t}$ the branch of $f^{-1}$
from $\widehat V_{i,t}$ to $\widehat V_{i,t+1}$ for $t=0$, $...$, $t_i+N-1$.
Let us denote by $H_i$ the family of all these branches.
By construction the family of domains
\[
\cU\eqdef
\widehat {\cU}_1 \cup \widehat {\cU}_2\cup\bigcup_{i=1,2}\bigcup_{t=1,...,t_i+N-1}
\widehat V_{i,t}
\]
and the family of maps $G\eqdef \widehat G_1\cup \widehat G_2 \cup H_1\cup H_2$ form our
desired GDS with a graph $\Gamma=\Gamma(\cU,G)$ satisfying the SSC and hence defines an $f$-ECR set $\Lambda\subset J$ that containes $\Lambda_1\cup\Lambda_2$.

Finally notice that  this system has topologically transitive limit set $\Lambda$ since its
transition
graph $\Gamma$ is transitive. This follows from the assumption that due to
topological transitivity of $f|_{\Lambda_i}$ the graphs $\Gamma_i$ are transitive
and from the construction of  the bridges.
\end{proof}

\subsection{Hyperbolic subsystems and approximation of pressure}\label{sec:hypsubpress}

Our approach is to ``exhaust'' the Julia set $J$ by some family of subsets $\Lambda_m\subset J$ and to show that the corresponding pressure functions converge towards the pressure of $f|_J$.
 In particular, in order to be able to conclude convergence of associated spectral quantities, it is crucial that each such $\Lambda_m$ is an invariant uniformly expanding and topologically transitive set.

We start by stating a classical result from Pesin-Katok theory. In
follows for example from~\cite[Theorems 10.6.1 and 11.2.3]{PUbook}.
Recall that an  iterated function system (IFS) is a GDS that is given by a complete graph.

\begin{lemma}\label{l0}
For every ergodic $f$-invariant measure $\mu$ that is supported on
$J$ and has a positive Lyapunov exponent, for every continuous
function $\phi\colon J\to \bR$
 and for every $\varepsilon>0$,
there exist an integer $n>0$ and an $f^n$-ECR set $\Lambda\subset
J$ that is topologically transitive and a limit set of an IFS,
such that
$\dim_{\rm H}\Lambda\ge \dim_{\rm H}\mu-\varepsilon$
\begin{equation}\label{pkeq}
P_{f^n|\Lambda}(S_n\phi)\ge h_\mu(f^n)+ n \int \phi\,d\mu -n
\varepsilon,
\end{equation}
where we use the notation $S_n\phi(x) \eqdef
\phi(x)+\phi(f(x))+\ldots+\phi(f^{n-1}(x))$, and in particular
\begin{equation} \label{pkeq2}
P_{f|\bigcup_{k=0}^{n-1} f^k(\Lambda)}(\phi)\ge h_\mu(f)+  \int \phi\,d\mu -
\varepsilon.
\end{equation}
\end{lemma}

Our aim is to apply Lemma~\ref{l0} to potentials $\phi=\varphi_d$
and to use the resulting ECR sets to construct a sequence of
$f^{a_m}$-ECR sets $\Lambda_m$ on which the pressure function
$\frac{1}{a_m}P_{f^{a_m}|\Lambda_m}(S_{a_m}\varphi_d)$ converges pointwise to
$P(\varphi_d)$. If the ECR sets generated by Lemma~\ref{l0} are
pairwise disjoint, we can simply build bridges between such sets
applying Lemma~\ref{bridge}. In the general situation we start
with the following lemma.

\begin{lemma} \label{bridge2}
Let $\Lambda$ be a topologically transitive $f^n$-ECR set. Let
$\phi_i\colon\Lambda\to\bR$ be a finite number of  continuous functions. Then for any open
disc $D$ intersecting $\Lambda$ and for any $\varepsilon>0$ one
can find a set $\Lambda'\subset D\cap \Lambda$ and a natural
number $\ell>0$ such that $\Lambda'$ is an $f^\ell$-ECR set and for every $\phi_i$
satisfies
\[
\frac{1}{\ell} P_{f^\ell|\Lambda'}(S_\ell\phi_i)
\geq \frac{1}{n}P_{f^n|\Lambda}(S_n\phi_i) - \varepsilon.
\]
\end{lemma}

\begin{proof}
Consider in $\Lambda$ a clopen set $C$ contained in $D$. Let $x\in
C$. Let us choose $N$ large enough such that in the case that $\ell\geq N$ and
$f_x^{-\ell}(C)\cap C \neq \emptyset$ the pullback satisfies
 $f_x^{-\ell}(C)\subset C$. Note that it is enough to take
\[
N>\frac{\log \diam \Lambda}{\rho_\Lambda (C, \Lambda\setminus
C)\log\lambda},
\]
where $\lambda$ is the expanding constant on $\Lambda$ in the
metric $\rho_\Lambda$ defined as in the proof of Lemma~\ref{lem:1111}. Note that
the topological transitivity of $f^n|_{\Lambda}$ implies that
every pullback of $C$ can be continued by a bounded number of
consecutive pullbacks until it hits  $C$, say this number is
bounded by a constant $N'$. This way we obtain an IFS for
$f^\ell$, where $N\le \ell\le N+N'$, with its limit set $\Lambda'$
contained in $C$.

Recall the equivalent definition of tree pressure established in~\cite[Theorems A, A.4]{PrzRivSmi:04}.
Due to topological transitivity in the definition of pressure we can consider separated sets  that are contained in the set of preimages $f^{-N}(x)$. Therefore, for any $\phi_i$ the pressures with respect to $f|_{\Lambda'}$ and with respect to $f|_{\Lambda}$ differ by at most $O(\frac {N'} N)$ from each other. As $N$
can be chosen arbitrarily big, this proves the lemma.
\end{proof}

The following approximation results are fundamental for our approach.

\begin{proposition}\label{decaf}
Assume that $f$ is non-exceptional or that $f$ is exceptional but
$\Sigma\cap J=\emptyset$. Then there exists a sequence $\{a_m\}_m$ of positive integers and a sequence $\Lambda_m\subset J$ of $f^{a_m}$-invariant uniformly expanding
topologically transitive sets such that for every $d\in\bR$, we
have
\begin{equation}\label{micha}
    P(\varphi_d)=\lim_{m\to\infty} \frac 1 {a_m} P_{f^{a_m}|\Lambda_m}(S_{a_m}\varphi_d)
    = \sup_{m\ge 1} \frac 1 {a_m} P_{f^{a_m}|\Lambda_m}(S_{a_m}\varphi_d).
\end{equation}
For every $\alpha\in (\alpha^-, \alpha^+)$ we have
\begin{equation}\label{fcon}
     F(\alpha)
    = \lim_{m\to\infty}F_m(\alpha)=     \sup_{m\ge 1}F_m(\alpha)
\end{equation}
and
\begin{equation}\label{alla}
    \lim_{m\to \infty} \alpha_m^- = \inf_{m\ge 1}\alpha_m^-=\alpha^- ,
\quad
\lim_{m\to \infty} \alpha_m^+ = \sup_{m\ge 1}\alpha_m^+=\alpha^+,
\end{equation}
where $F_m$ and $\alpha_m^\pm$ are defined as in~\eqref{def:Fa}
and~\eqref{yu} but with  $\frac{ 1 }{a_m}
P_{f^{a_m}|\Lambda_m}(S_{a_m}\phi)$ instead of $P_{f|J}(\phi)$.
\end{proposition}

\begin{proof}
To prove \eqref{micha} it is enough to construct an $f^{a_m}$-ECR
set $\Lambda_m\subset J$ such that
\begin{equation} \label{brez}
\frac 1 {a_m} P_{f^{a_m}|\Lambda_m}(S_{a_m}\varphi_d) \geq
P(\varphi_d) - \frac 1 m
\end{equation}
for all $d\in [-m,m]$. Recall that we have~\eqref{Phyp}. As
$d\mapsto P(\varphi_d)$ is uniformly Lip\-schitz continuous, we only need to
check~\eqref{brez} for a finite number of potentials
$\phi_i=\varphi_{d_i}$.
Given $m$, we apply Lemma~\ref{l0} to potentials $\phi_i$,
obtaining a family of $f^{n_i}$-ECR sets $\Lambda_{m,i}$ on which
the pressure $\frac {1}{ n_i}
P_{f^{n_i}|\Lambda_{m,i}}(S_{n_i}\phi_i) \ge P(\phi_i)-\frac{1}{2 m}$.
We then apply Lemma~\ref{bridge2} to construct a family of
pairwise disjoint $f^{\ell_i}$-ECR sets $\Lambda_{m,i}'$ that 
satisfy $\frac 1 {\ell_i} P_{f^{\ell_i}
|\Lambda_{m,i}'}(S_{\ell_i}\phi_i) \ge P(\phi_i)-\frac 1 m$. Those
sets $\Lambda_{m,i}'$ are also disjoint $f^{a_m}$-ECR sets for
$a_m=\prod_i \ell_i$ and by our assumption contain non-immediately
post-critical points. Hence we can consecutively apply
Lemma~\ref{bridge} to them. We obtain an $f^{a_m}$-ECR set
satisfying~\eqref{brez} for all $\phi_i$. This
proves~\eqref{micha}.

As $d\mapsto P(\varphi_d)$ and $\alpha\mapsto\alpha F(\alpha)$
form a Legendre pair,~\eqref{fcon} and~\eqref{alla} follow from
\eqref{micha} by a result of Wijsman~\cite{Wij:66}.
\end{proof}

\begin{remark}
It is enough for us to work with hyperbolic sets for some
iterations of $f$ instead of hyperbolic sets for $f$ (in other
words, to use~\eqref{pkeq} instead of~\eqref{pkeq2}). However, notice that we could
instead also extend the set $\bigcup_{i=0}^{a_m-1} f^i(\Lambda_m)$ to
an $f$-ECR set using Lemma~\ref{lem:1111}.
\end{remark}

\subsection{Conformal measures}\label{sec:conf}

The dynamical properties of any measure $\nu$ with respect to
$f|_J$ are captured through its Jacobian. The \emph{Jacobian} of
$\nu$ with respect to $f|_J$ is the (essentially) unique function
${\rm Jac}_\nu f$ determined through
\begin{equation}\label{conf}
\nu(f(A)) = \int_A{\rm Jac}_\nu f \,d\nu
\end{equation}
for every Borel subset $A$ of $J$ such that $f|_A$ is injective.
In particular, its existence yields the absolute continuity $\nu
\circ T \prec \nu$.

A probability measure $\nu$ that satisfies
\[
{\rm Jac}_\nu f=e^{P(\varphi_d)-\varphi_d}
\]
is called \emph{$e^{P(\varphi_d)-\varphi_d}$-conformal measure}.
If $d\geq 0$ then one can always find a
$e^{P(\varphi_d)-\varphi_d}$-conformal measure $\nu_d$ that is
positive on each open set intersecting $J$, see~\cite{PrzRivSmi:04}. When $d<0$ such a measure can always be found
if $f$ is not an exceptional map or if $f$ is exceptional but
$\Sigma\cap J=\emptyset$ (see~\cite[Appendix A.2]{PrzRivSmi:04}).

\subsection{Hyperbolic times and conical limit points}\label{sec:hypinst}

When $f|_J$ is not uniformly expanding, we can still observe a
slightly weaker form of non-uniform hyperbolicity. We recall two
concepts that have been introduced.

A number $n\in\bN$ is called a \emph{hyperbolic time} for a point
$x$ with exponent $\sigma$ if
\[
\lvert (f^k)'(f^{n-k}(x))\rvert \ge e^{k\sigma} \text{ for every
}1\le k \le n.
\]
It is an immediate consequence of the Pliss lemma (see, for
example,~\cite{Pli:72}) that for a given point $x\in J$, for any
$\sigma<\overline\chi(x)$ there exist infinitely many hyperbolic
times for $x$ with exponent $\sigma$.

We denote by
\[
\dist g|_Z \eqdef \sup_{x,y\in Z}\frac{\lvert g'(x)\rvert}{\lvert
g'(y)\rvert}
\]
the maximal distortion of a map $g$ on a set $Z$.
After~\cite{4aut}, we will call a point $x\in J$ {\it conical} if
there exist a number $r>0$, a sequence of numbers $n_i\nearrow \infty$ and a sequence
$\{U_i\}_i$ of neighborhoods of $x$ such that $f^{n_i}(U_i)\supset
B(f^{n_i}(x),r)$ and that $\dist f^{n_i}|_{U_i}$ is bounded
uniformly in $i$.

\section{On the completeness of the spectrum}\label{sec:3.5}

In the following two lemmas we will investigate which numbers can occur at all as upper/lower Lyapunov exponents.

\begin{lemma}\label{lem:finit1}
We have
\[
\left\{ x\in J\colon \overline\chi(x)>\alpha^+\right\} =\emptyset.
\]
If $J$ does not contain any critical point of $f$ then we have
\[
\left\{x\in J\colon\underline\chi(x)< \alpha^- \right\} =\emptyset.
\]
\end{lemma}

\begin{proof}
Consider an arbitrary $x\in J$ and a sequence $n_i\nearrow\infty$ such
that
\[
\lim_{i\to\infty} \frac 1 {n_i} \sum_{j=0}^{n_i-1} \log\, \lvert
f' (f^j(x))\rvert = \overline\chi(x)
\]
and
\[
\mu_i\eqdef\frac 1 {n_i} \sum_{j=0}^{n_i-1} \delta_{f^j(x)}\, \to
\mu
\]
in the weak$\ast$ topology. The limit measure $\mu$ is $f$-invariant,
\cite[Theorem 6.9]{Wal}.
Define
\[
g_N\eqdef\max\{\log\, \lvert f'\rvert, -N\} .
\]
Notice that $\{g_N\}_N$ is a monotonically decreasing sequence of continuous functions
that
converge pointwise to $\log\,\lvert f'\rvert$. Hence we obtain
\[
\overline\chi(x)
\le \lim_{N\to\infty}\int g_N\,d\mu
=\int\log\,\lvert f'\rvert\,d\mu
=\chi(\mu)\le \alpha^+,
\]
where the equality follows from the Lebesgue monotone
convergence theorem. This proves the first statement.

The second statement follows simply from the fact that $\log \,\lvert f'\rvert$ is
continuous on $J$ if $f$ has no critical points in $J$.
\end{proof}

\begin{remark}{\rm
We remark that the same method of proof gives a slightly stronger result than the fact
that
$\overline\chi(x)\le\alpha^+$ for every $x$. Namely we have
\[
\lim_{n\to\infty}\sup_{z\in\bC} \frac 1 n \log\, \lvert (f^n)'(z)\rvert \le \alpha^+,
\]
see~\cite[Proposition~2.3.~item~2]{PrzRiv}.

Recall that on a set of total probability we have $\chi(x)\ge
0$~\cite{Prz:93}. In fact, a better estimate can be given for any
Lyapunov regular point $x$ (compare the proof of~\cite[Proposition
4.1]{PrzRivSmi:03}).
}\end{remark}

\begin{lemma}\label{lem:finit2}
If $x\in J$ has a finite Lyapunov exponent $\chi(x)$ then $\chi(x)\geq \alpha^-$, that
is, we have
\[
 \left\{x\in J\colon -\infty<\chi(x)<\alpha^- \right\} = \emptyset.
\]
If there are no critical points in $J$ then $\cL(-\infty)$ is
empty. If there is only one critical point in $J$ then
$\cL(-\infty)$ consists only of this critical point and its
preimages.
\end{lemma}

\begin{proof}
Let $x\in J$ be a Lyapunov regular point with exponent $\chi(x)$
and assume that $\chi(x) < \alpha^-$. It is enough for us to prove that there
exists a periodic point with Lyapunov exponent arbitrarily
close to $\chi(x)$, the contradiction will then follow from the
definition of $\alpha^-$.

Note first that if $\chi(x)$ exists and is finite then $\frac 1 n
\log \,\lvert (f^n)'(x)\rvert$ must be a Cauchy sequence and hence satisfies
\begin{equation}\label{xuxu}
\liminf_{n\to\infty} \frac 1 n \log \rho(f^n(x), {\rm Crit}) =0.
\end{equation}
By \cite[Corollary to Lemma~6]{Prz:93}, we then can conclude that $\chi(x)\ge
0$.
Choose now a small number $\delta>0$. 
Let $n$ be a hyperbolic time for $x$ with exponent $-\delta/2$ 
 (recall that $x$ has infinitely many hyperbolic times with that exponent, and hence that $n$ can be chosen arbitrarily big). 
Because of~\eqref{xuxu}, there exists an integer $n_0>0$ such that for all $k\geq n_0$ we have
\begin{equation} \label{cos0}
\rho(f^k(x), {\rm Crit}) > \exp(-k\delta)
\end{equation}
and we can assume that $n>n_0$. 

We start with a construction that is standard in Pesin theory. For each $k =0$, $\ldots$, $n-n_0-1$ we define
\[
B_k\eqdef B\left(f^{n-k}(x), e^{(-(d+1)n+ k)\delta}\right),
\]
where $d$ is the greatest degree of a critical point of $f$. This
way we define a sequence of balls that are centered at points of
the backward trajectory of $f^n(x)$ and that have  diameters shrinking
slower than the derivative of $f^{-k}$ along this branch and at the same time have diameters much smaller than their distance from any critical point.
As we have $f^{n-k}(x)\in B_k$,~\eqref{cos0} implies that for $n$ big
enough for any $k<n-n_0$ the set $f^{-1}_{f^{n-k-1}(x)}(B_k)$ does not
contain any critical point and that
\begin{equation} \label{cos}
\log \dist f^{-1}_{f^{n-k-1}(x)}|_{B_k} \leq K_1 \frac {\diam B_k}
{\rho(f^{n-k}(x), {\rm Crit})},
\end{equation}
where $K_1$ is some constant that depends only on $f$.
\\[0.2cm]
\textbf{Claim:} 
If $n$ is sufficiently big then for any $k=0$, $\ldots$, $n$
the map $f^k$ is univalent  and has bounded distortion $\le \exp(\delta/2)$ on the set $f^{-k}_{f^{n-k}(x)}(B_0)$
and satisfies $f^k(B_k)\supset B_0$. 
\\[-0.2cm]

To prove the above claim, note
that the ball $B_{n-k}$ 
shrinks as $n$ increases. 
Hence, it is enough to prove the statement for $k< n-n_0$.
 The statement for the initial finitely many steps $k=n-n_0$, $\ldots$, $n$ is then
automatically provided $n$ is big enough.
Let us assume that $n$ is sufficiently big such that also
\[
K_1 (n-n_0) e^{-dn\delta} < \frac \delta 2.
\]
By construction of the family $\{B_k\}_k$ and by~\eqref{cos0},  for each $\ell\ge 1$ we have
\[
\sum_{k=0}^{\ell-1} \frac {\diam B_k} {\rho(f^{n-k}(x), {\rm Crit})}
< \sum_{k=0}^{\ell-1}\frac{ e^{(-(d+1)n+k)\delta}}{ e^{(-n+k)\delta}}
= \ell{e^{-dn\delta}}.
\]
Hence, if for $\ell\leq n-n_0$ we have $f^k(B_k)\supset B_0$ for every
$k=1$, $\ldots$, $\ell-1$ then~\eqref{cos} implies
\[
\log \dist f^{-\ell}_{f^{n-\ell}(x)} |_{B_0} \leq K_1 \ell
e^{-dn\delta} < \frac \delta 2.
\]
On the other hand, recall that $n$ is a hyperbolic time for $x$ with
exponent $-\delta/2$. Hence, if
\[
\log \dist f^{-\ell}_{f^{n-\ell}(x)} |_{B_0} < \frac \delta 2
\]
then $f^\ell(B_\ell)\supset B_0$. The above claim now follows by induction over
$\ell$.

Let us consider the set
\[
E\eqdef \bigcap_{k\ge1}\bigcup_{\ell > k} E_\ell, \quad\text{ where }
E_\ell\eqdef
\bigcup_{j=1}^{\ell} B\big(f^j({\rm Crit}), e^{-2d\ell\delta}\big).
\]
Notice that $B_0\setminus E_n$ is nonempty whenever the hyperbolic time $n$ is big enough.
For such $n$, let  $y\in B_0\setminus E_n$. From our distortion estimations for
$f^{-n}_{x}|_{B_0}$ in the Claim we obtain
\begin{equation} \label{pgui}
\log \frac {\left|(f^{-n}_x)'(f^n(x))\right|}
{\left|(f^{-n}_x)'(y)\right|} < \frac \delta 2.
\end{equation}

In the remaining proof we will follow closely techniques in~\cite[Section 3]{PrzRivSmi:03}.
Let us fix some arbitrary $f$-invariant uniformly expanding set $\Lambda\subset J$ that has positive Hausdorff dimension (the existence of such a set follows for example from Lemma~\ref{l0}). As $\Crit$ is a finite set, we have $\dim_{\rm \Lambda} E=0$ and we can find a point $z\in \Lambda\setminus E$. In particular $z\in \Lambda\setminus E_n$ for $n$ large.
Note that 
in particular for every $n$ large enough we have %
\begin{equation}\label{cacho}%
\rho(z,f^n(\Crit))> e^{-2dn\delta}.%
\end{equation} and hence %
on the disk $B(z,e^{-2dn\delta})$ any pull-back  $f^{-n}$ is univalent. 

 By~\cite[Lemma 3.1]{Prz:99}, there exists a number  $K_2>0$ depending only on $f$ and a sequence of disks $\{D_i\}_{i=1,\ldots,K}$ such that $\bigcup_{i=1}^KD_i$ is connected, $y$ is the center of disk $D_1$, $z$ is
the center of disk $D_K$, 
\[
\rho\Big(D_i, \bigcup_{j=1}^{n} f^j ({\rm Crit})\Big) \geq
\diam D_i,
\]
and the number of disks is bounded by $K\le K_2(n \delta)^{1/2}$.
By the Koebe distortion lemma, for each branch of $f^{-n}$ and for every $D_i$ we have
\[
\dist f^{-n}|_{D_i} < K_3,
\]
where $K_3>0$ is some constant.
 Hence, in particular
\[
\left\lvert\log \,\lvert(f^{-n}_w)'(z)\rvert - \log\,
\lvert(f^{-n}_{x})'(y)\rvert\right\rvert \leq K \log K_3
\]
for some $w\in f^{-n}(z)$. Together with~\eqref{pgui}, this
implies that
\[\begin{split}
\frac 1 n\Big\lvert \log\,\lvert(f^n)'(x)\rvert  -  \log\,\lvert(f^n)'(w)\rvert \Big\rvert 
&\leq \frac 1 n \left(\frac\delta 2 + K\log K_3 \right)\\
&\leq \frac 1 n \left(\frac\delta 2 + K_2(n\delta)^{1/2}\log K_3 \right) 
<\frac \delta 2,
\end{split}\]
whenever $n$ has been chosen large enough.

By hyperbolicity of $f|_\Lambda$, there exist positive  constants $\Delta$ and $K_4$ such that for all $\ell\ge0$ the map $f^{-\ell}$ is univalent and has bounded distortion $\le K_4$ on $B(f^\ell(z),\Delta)$. 
Recall that the Julia set $J$ has the property that there exists $m=m(\Delta)>0$ such that the image $f^m(B(f^\ell(z),\Delta)\cap J)$ is equal to $J$.

Let $\ell$ be the smallest positive integer such that $\lvert (f^\ell)'(z)\rvert\ge K_4\Delta e^{2d(n+m)\delta}$. Hence $f^{-\ell}_z(B(f^\ell(z),\Delta))\subset B(z,e^{-2d(n+m)\delta})$ and it is easy to show that $\ell\le K_5+K_6(n+m)\delta$ for some constants $K_5$, $K_6$. 
 So in particular, one of the preimages $f^{-(n+m)}(z)$ is in $B(f^\ell(z),\Delta)$ and the corresponding pull-back $W\eqdef f^{-(n+m)}(B(z,e^{-2d(n+m)\delta}))$ satisfies
$W\subset B(f^{\ell}(z),\Delta)$. It follows from~\eqref{cacho} that $f^{-(n+m+\ell)}\colon W\to B(z,e^{-2(n+m)\delta})$ is univalent and hence that there is a repelling fixed point $p=f^{n+m+\ell}(p)$ in $B(z,e^{-(n+m+\ell})$. 
 
Using the above, the Lyapunov exponent of $p$ can be  estimated by
\[\begin{split}
\chi(p)&=\frac{1}{n+m+\ell}\log\,\lvert (f^{n+m+\ell})'(p)\rvert \\
	&\le \frac{1}{n}\log\left\lvert(f^n)'(w)\right\rvert\frac{1}{1+(m+\ell)/n}
	+ \frac{m+\ell}{n+m+\ell}\sup \log\,\lvert f'\rvert\\
	&\le \left(\chi(x)+\frac \delta 2\right)\frac{1}{1+(m+\ell)/n}
	+ \frac{m+\ell}{n+m+\ell}\sup \log\,\lvert f'\rvert 
	.
\end{split}\]
Now recall that the hyperbolic time $n$ can be chosen arbitrarily large and that $\delta$ was chosen arbitrarily small. This proves the first claim of the lemma.

If there are no critical points in $J$ then $\log \,\lvert f'\rvert$ (and
hence $\chi$) is bounded from below. The remaining claim of the
lemma follows from~\cite[Lemmas 2.1 and 2.3]{DenPrzUrb}.
\end{proof}

\section{Upper bounds for the dimension}\label{sec:up}

In this section we will derive upper bounds for the Hausdorff dimension of the level sets.
We first start with a particular case.

\subsection{No critical points in $J$}\label{sec:nocrit}

We study the particular case that there is no critical point in the Julia set $J$
(though, parabolic points in $J$ are allowed).

We start by taking a more general point of view and investigate those points that have
some least degree of hyperbolicity. In terms of Lyapunov exponents, this concerns points
$x$ with
$\overline\chi(x)>0$.
We begin our analysis by presenting a simple lemma, that will be useful shortly after.

\begin{lemma} \label{lem:sequ}
Let $\{a_n\}_n$ be a sequence of real numbers such that
$\{a_{n+1}-a_n\}_n$ converges to zero but $\{a_n\}_n$ does not have a
limit. Then for any natural number $r$ and for any number $q\in
[\liminf_{n\to\infty} a_n, \limsup_{n\to\infty}  a_n]$ there exists a subsequence
$n_k\nearrow\infty$  such that
\[
\lim_{k\to \infty} a_{n_k} = q
\]
and for every $k$ we have
\[
a_{n_k}<a_{n_k+r}.
\]
\end{lemma}

\begin{proof}
We will restrict our hypothesis to the case that $r=1$. The general case then follows
from considering the  subsequence $\{a_{rn}\}_n$.

First note that every number  $q\in [\liminf a_n,\limsup a_n]$ is the limit of
some subsequence of $\{a_n\}_n$. Hence, if $q\neq\liminf a_n$ then
for every $\varepsilon >0$ there must exist infinitely many numbers
$m=m(\varepsilon)$ such that
\[
a_m<q-\varepsilon<a_{m+1}.
\]
Similarly, if $q\neq\limsup a_n$ then for every $\varepsilon >0$
there must exist infinitely many numbers $m=m(\varepsilon)$ such that
\[
a_m<q+\varepsilon<a_{m+1}.
\]
Since we assume that the sequence $\{a_n\}_n$ does not have a
limit, $q$ must satisfy one of the abovementioned properties.
Hence, if we choose a decreasing family $\{\varepsilon_k\}_k$ and
for each $\varepsilon_k$ one of the corresponding numbers
$n_k=m(\varepsilon_k)$, we obtain
\[
\lvert a_{n_k}-q \rvert \leq \varepsilon_k + \lvert a_{n_k}-a_{n_k-1}\rvert  \to 0.
\]
The second part of the assertion is immediately satisfied.
\end{proof}

Lemma~\ref{lem:sequ} will help us to establish some bounded
distortion properties. The following result implies in particular
that every point $x$ with $\overline\chi(x)>0$ is conical (recall
the definition of a conical point in Section~\ref{sec:hypinst}).

\begin{lemma} \label{lem:bdp}
Assume that $J$ does not contain any critical points of $f$. Let
$x\in J$ be a point with $\overline{\chi}(x)>0$. Then there exists a number $K>0$ such
that for every $q\in [\underline{\chi}(x),
\overline{\chi}(x)]\setminus \{0\}$ there exists a number $\delta>0$ and a
sequence $n_k\to \infty$ such that
\begin{itemize}
\item[i)] $\displaystyle \lim_{k\to \infty} \frac 1 {n_k} \log \,\lvert
(f^{n_k})'(x)\rvert =
q$,\\[-0.1cm]
\item[ii)] $\displaystyle \dist f^{n_k}|_{f_x^{-n_k}(B(f^{n_k}(x), \delta))} <
K.$
\end{itemize}
Here $K$ is a universal constant, while $\delta$ depends on the number $q$ but not
on the point $x$.
\end{lemma}

\begin{proof}
As $J$ does not contain any critical point, the only accumulation
points in $J$ of the orbit of some critical point can be parabolic points.
Let $r$ be the least common multiplier of the
periods of all parabolic points in $J$.

Given a number $q\in [\underline{\chi}(x),
\overline{\chi}(x)]\setminus \{0\}$, there exists a number
$\delta_0>0$ such that if $y\in J$ is 
$\delta_0$-close to some
 parabolic point then
\[
\frac 1 r \log \,\lvert (f^r)'(y)\rvert \leq \frac q 2.
\]
Further, there exists a number $\delta_1>0$ such that if $z\in J$ is
$\delta_0$-away from any
parabolic point then no orbit of a critical
point passes through $B(z, 2\delta_1)$.

To prove the claimed property, it will suffice to find a sequence
$\{n_k\}_k$ for which i) is satisfied and for which $f^{n_k}(x)$ is
in distance at least $\delta_0$ from any parabolic point. Indeed,
in such a situation the backward branch of
$f_x^{-n_k}(B(f^{n_k}(x), 2\delta_1))$ will not catch any critical
point and the distortion estimations ii) will follow from the
Koebe distortion lemma.

If $x$ is a Lyapunov regular point and
$q=\lim_{n\to\infty}\frac{1}{n}\log\,\lvert (f^n)'(x)\rvert$ is
its Lyapunov exponent (which, by our assumptions, must be
positive) then the claim i) is automatically satisfied. In this
case we just need to choose $n_k$ such that $f^{n_k}(x)$ is far
away from any parabolic point. Note that there must be infinitely
many such times $n_k$ because otherwise the Lyapunov exponent at
$x$ would be no greater than $q/2$.

If $x$ is not a Lyapunov regular point, then we apply Lemma~\ref{lem:sequ} to the sequence
\[
a_n=\frac 1 n \log \,\lvert (f^n)'(x)\rvert.
\]
Notice that $\lim_{n\to\infty} (a_{n+1}-a_n)=0$ is satisfied,
because there are no critical points in $J$ and hence $\lvert
f'\rvert$ is uniformly bounded. Hence, from the first assertion of
Lemma~\ref{lem:sequ} we  obtain a sequence $\{n_k\}_k$ that
satisfies i). Notice that we have
\[
a_{n_k+r}
= \frac {n_k} {n_k+r} a_{n_k} + \frac 1 {n_k+r} \log\,\lvert (f^r)'(f^{n_k}(x))\rvert
\leq \frac {n_k} {n_k+r} a_{n_k} +\frac {r}{n_k+r} \frac q 2
\]
whenever $f^{n_k}(x)$ is 
$\delta_0$-close to some
parabolic point. This inequality cannot be true for big $n_k$
(when $a_{n_k}$ is already close to $q$) because of the second
part of assertion of Lemma \ref{lem:sequ}. This proves that for
any time $n_k$ large enough the point $f^{n_k}(x)$ is
$\delta_0$-away from any parabolic point.
\end{proof}

We are now prepared to prove an upper bound for the Hausdorff
dimension of the level sets under consideration. To start with the
most general approach that will be needed in the subsequent
analysis, we first study a set of points $x$ for which
$\overline\chi(x)>0$ and for which additionally the Lyapunov
exponent (possibly with respect to some subsequence of times) is
guaranteed to be within a given interval $[\alpha,\beta]$. Let us
first introduce some notation. Given $0\le\alpha\le\beta$,
$\beta>0$, let
\begin{equation} \label{cl}
\ccL(\alpha,\beta) \eqdef
 \{x\in J\colon \underline\chi(x)\leq\beta, \, \overline\chi(x)\geq\alpha, \,
\overline\chi(x)>0\}.
\end{equation}

\begin{proposition}\label{prop:1}
Assume that $J$ does not contain any critical points of $f$. For
every $\beta>0$ and $0\le \alpha\le \beta$ we have
    \[
    \dim_{\rm H}\widehat\cL(\alpha,\beta)\le
    \max_{\alpha\le q\le \beta} F(q).
    \]
    If $\alpha>\alpha^+$ or $\beta<\alpha^-$ then $\widehat\cL(\alpha,\beta)=\emptyset$.
\end{proposition}

\begin{proof}
By Lemma \ref{lem:bdp}, for every point $x\in \ccL(\alpha, \beta)$
there exist a number $q=q(x)\in[\alpha,\beta]\setminus \{0\}$, a
number $\delta>0$, and a sequence $\{n_k\}_k$ of numbers such that
\begin{equation} \label{eqn:gh}
\lim_{k\to\infty}\frac{1}{n_k}\log\,\lvert (f^{n_k})'(x)\rvert=q.
\end{equation}
and
\begin{equation} \label{est2}
2\delta\,\lvert(f^{n_k})'(x)\rvert^{-1} K^{-1}\le \diam
f^{-n_k}_x(B(f^{n_k}(x),\delta)) \le 2\delta
\,\lvert(f^{n_k})'(x)\rvert^{-1}K.
\end{equation}

Recall that for every $d\in\bR$ there exists a
$\exp\left(P(\varphi_d)-\varphi_d\right)$-conformal measure
$\nu_d$ that gives positive measure to any open set (see Section~\ref{sec:conf}). Hence there
exists $c=c(\delta)>0$ such that for every $n_k$ we have $c\le
\nu_d(B(f^{n_k}(x),\delta))\le 1$. Using again the distortion
estimates, we can conclude that
\begin{multline}\label{huch}
    c K^{-d} \,e^{-n_k P(\varphi_d)} \lvert (f^{n_k})'(x)\rvert^{-d} \\
\le \nu_d(f^{-n_k}_x(B(f^{n_k}(x),\delta)))\le
    K^d e^{-n_k P(\varphi_d)}\lvert (f^{n_k})'(x)\rvert^{-d},
\end{multline}
which implies that
\[
\lim_{k\to\infty}\frac{1}{n_k}\log\,
\nu_d(f^{-n_k}_x(B(f^{n_k}(x),\delta))) = -P(\varphi_d)-d\,q
                  \]
and in particular that this limit exists. Hence, there exists $N>0$ such that
for every $n_k\ge N$ we have
\[\begin{split}
 \nu_d \big(f^{-n_k}_x(B  (f^{n_k}(x),&\delta))\big)\\
&\ge e^{-n_k(P(\varphi_d)+dq+d\delta)}\\
&\ge e^{-n_kP(\varphi_d)} \lvert (f^{n_k})'(x)\rvert^{-d}\\
&\ge  K^{-2d} e^{-n_kP(\varphi_d)} \left(\diam
f^{-n_k}_x(B(f^{n_k}(x),\delta))\right)^d
\left(\frac{1}{2\delta}\right)^d.
\end{split}\]
Here we used Lemma~\ref{lem:bdp} to obtain the
last inequality. Applying~\eqref{eqn:gh} and~\eqref{est2} we yield
\begin{equation}\label{esto3}
\underline d_{\nu_d}(x) \le \frac{P(\varphi_d)}{q}+d .
\end{equation}

For any $q_0\in [\alpha, \beta]$, $q_0>0$, and $\varepsilon>0$
there exist a  small interval $(q_1,q_2)$, $q_1>0$, containing
$q_0$ and a number $d\in\bR$ such that 
for all $q\in (q_1,q_2)$
\begin{equation} \label{ala1}
\frac{1}{q} P(\varphi_d) + d <
\begin{cases}
F(q) + \varepsilon  &\text{ if }F(q)\neq -\infty ,\\
 -100               &\text{ if }F(q)=-\infty.\end{cases}
\end{equation}
We can choose a countable family of intervals $\{(q_1^{(i)},
q_2^{(i)})\}_i$ covering $[\alpha,\beta]\setminus \{0\}$ and a
sequence of corresponding numbers $\{d_i\}_i$. Defining the
measure
\[
\nu \eqdef \sum_{i} 2^{-i} \nu_{d_i}
\]
we obtain
\[
\underline{d}_\nu(x) \le \inf_{i\ge1} \underline{d}_{\nu_{d_i}}(x) \le
\max_{\alpha\le q \le \beta} F(q) +\varepsilon,
\]
where the second inequality follows from~\eqref{ala1}. Applying
the Frostman lemma we obtain that
\[
\dim_{\rm H}\ccL(\alpha, \beta) \le \max_{\alpha \le q \leq \beta}
F(q) +\varepsilon
\]
Since $\varepsilon$ can be chosen arbitrarily small, this finishes
the proof of the first claim. The second claim was already proved
in Lemma~\ref{lem:finit1}.
\end{proof}

Note that for every $q\in[\alpha,\beta]$, $q>0$, we have
$\cL(\alpha,\beta)\subset \widehat\cL(q,q)$, which readily proves
the following result.

\begin{corollary}\label{cor:1}
    Under the hypotheses of Proposition~\ref{prop:1}, we have
    \[
    \dim_{\rm H}\cL(\alpha,\beta) \le \min_{\alpha\le q \le \beta}F(q)
    =\min\{F(\alpha),F(\beta)\}.
    \]
    In particular, for every $\alpha>0$, we have
    \[
    \dim_{\rm H}\cL(\alpha) \le F(\alpha).
       \]
\end{corollary}

\subsection{The general case}\label{sec:upgen}

We now consider the general case that there are critical
points inside the Julia set.

We need two technical results from the literature. The first one is
the following telescope lemma from~\cite{Prz:90}. Recall the
definition of hyperbolic times given in Section~\ref{sec:hypinst}.

\begin{lemma} \label{telescope}
Given $\varepsilon>0$ and $\sigma>0$, there exist constants
$K_1>
0$ and $R_1>
0$ such that the following is true. Given $x\in J$ with upper
Lyapunov exponent $\overline\chi(x)>\sigma$, for every number
$r<R_1$, for every $n\ge1$ being a hyperbolic time for $x$ with
exponent $\sigma$, and
for every $0\le k\le n-1$ we have
    \[
    \diam f^{-n+k}_{f^k(x)}\left(B(f^n(x),r)\right) \le r\,K_1 e^{-(n-k)(\sigma -
\varepsilon)} .
    \]
\end{lemma}

To formulate our second preliminary technical result we need the following construction,
see~\cite{GraSmi,PrzRivSmi:03,PrzRivSmi:04}.\\[0.2cm]
{\bf Pullback construction:} Fix some $n>0$ and let $y\in
J\setminus \bigcup_{i=1}^n f^i(\Crit)$. Fix some $R>0$ and let
$\{y_i\}_{i=1}^n$ be some backward trajectory of $y$, i.e. $y_0=y$
and $y_{i+1}\in f^{-1}(y_i)$ for every $i=1$, $\ldots$, $n-1$. Let
$k_1$ be the smallest integer for which
$f^{-k_1}_{y_{k_1}}(B(y,R))$ contains a critical point. For every
$\ell\ge 1$ let then $k_{\ell+1}$ be the smallest integer greater
than $k_\ell$ such that
$f^{-(k_{\ell+1}-k_\ell)}_{y_{k_{\ell+1}}}(B(y_{k_\ell},R))$
contains a critical point and so on. In this way, for each
backward branch $\{y_i\}_i$ we construct a sequence
$\{k_\ell\}_\ell$ that must have a maximal element not greater
than $n$. Let this element be $k$ and consider the set $Z$ of all
pairs $(y_k, k)$ built from all the backward branches of $f$ that
start from $y$. Let $N(y,n,R) \eqdef \# Z$. We have the following
estimate, see~\cite[Lemma 3.7]{PrzRivSmi:04} and~\cite[Appendix
A]{PrzRivSmi:03}.

\begin{lemma} \label{prs}
Given $\varepsilon > 0$, there exist  $K_2>0$ and  $R_2>0$ such
that for all $R\leq R_2$ we have
\[
N(y,n,R) < K_2 e^{n\varepsilon}
\]
uniformly in $y$ and $n$.
\end{lemma}

Recall the definition of conical points in
Section~\ref{sec:hypinst}. Using the above two lemmas we can now
show the following result.

\begin{proposition}\label{prop:wander}
The set of points $x\in J$ that are not conical and satisfy $\overline\chi(x)>0$
  has Hausdorff dimension zero.
\end{proposition}

\begin{proof}
Let us choose some numbers $\sigma >0$ and $\varepsilon>0$. Let
\[
r\eqdef\frac 1 {2K_1+4}\min\{R_1,R_2\},
\]
where $K_1$ and $R_1$ are constants given by Lemma~\ref{telescope}
and where $R_2$ is given by Lemma~\ref{prs}.

We can choose a finite family of balls $\{B_i\}_{i=1}^L$ of radius
$3r$ such that any ball of radius $2r$ intersecting $J$ must be
contained in one of the balls $B_i$. In the case that we can prove
existence of a sequence $\{n_i\}_i$ such that
$f^{-(n_i-k)}_{f^k(x)}(B(f^{n_i}(x), 2r))$ does not contain
critical points for any $0\le k\le n_i-1$,  the
Koebe distortion lemma will imply that $x$ is a conical point for $r$,
$\{n_i\}_i$, and $U_i= f^{-n_i}_x(B(f^{n_i}(x),r))$.

Let $G(m,\sigma)$ be the set of points $x\in J$ with upper
Lyapunov exponent greater than $\sigma$ for which with $r$ chosen
above for all $n>m$ the backward branch of $f^{-n}$ from ball
$B(f^n(x),2r)$ onto a neighborhood of $x$ will necessarily meet a
critical point, that is, for some $0\le k\le n-1$ we have
\[
f^{-(n-k)}_{f^k(x)}(B(f^n(x), 2r))\cap\Crit \ne\emptyset.
\]

First we claim that $\dim_{\rm H}G(m,\sigma)=0$. Let us denote by
$G(m,\sigma, n)$ the subset of $G(m,\sigma)$ for which $n>m$ is a
hyperbolic time with exponent $\sigma$. Recall that
$\overline\chi(x)>\sigma>0$ implies that there exist infinitely
many hyperbolic times for $x$ with exponent $\sigma$. Hence, we
have
\begin{equation} \label{eqn:prs}
G(m,\sigma) = \bigcap_{m_0\geq m} \bigcup_{n>m_0} G(m, \sigma, n).
\end{equation}
Let $x\in G(m, \sigma, n)$. Let $B_j=B(y,3r)$ be the ball that
contains $B(f^n(x),2r)$. We will apply the ``pullback
construction" for Lemma~\ref{prs} to the point $y$, the numbers
$n$, $R=\frac 1 2 \min\{R_1,R_2\}$, and to the backward branch
$f^{-n}_x$. And let $k=\max_\ell k_\ell$ and $y_k$ be given by the
pullback construction (compare Figure~\ref{Fig.reia}).
\begin{figure}[h]
         \includegraphics[width=7cm]{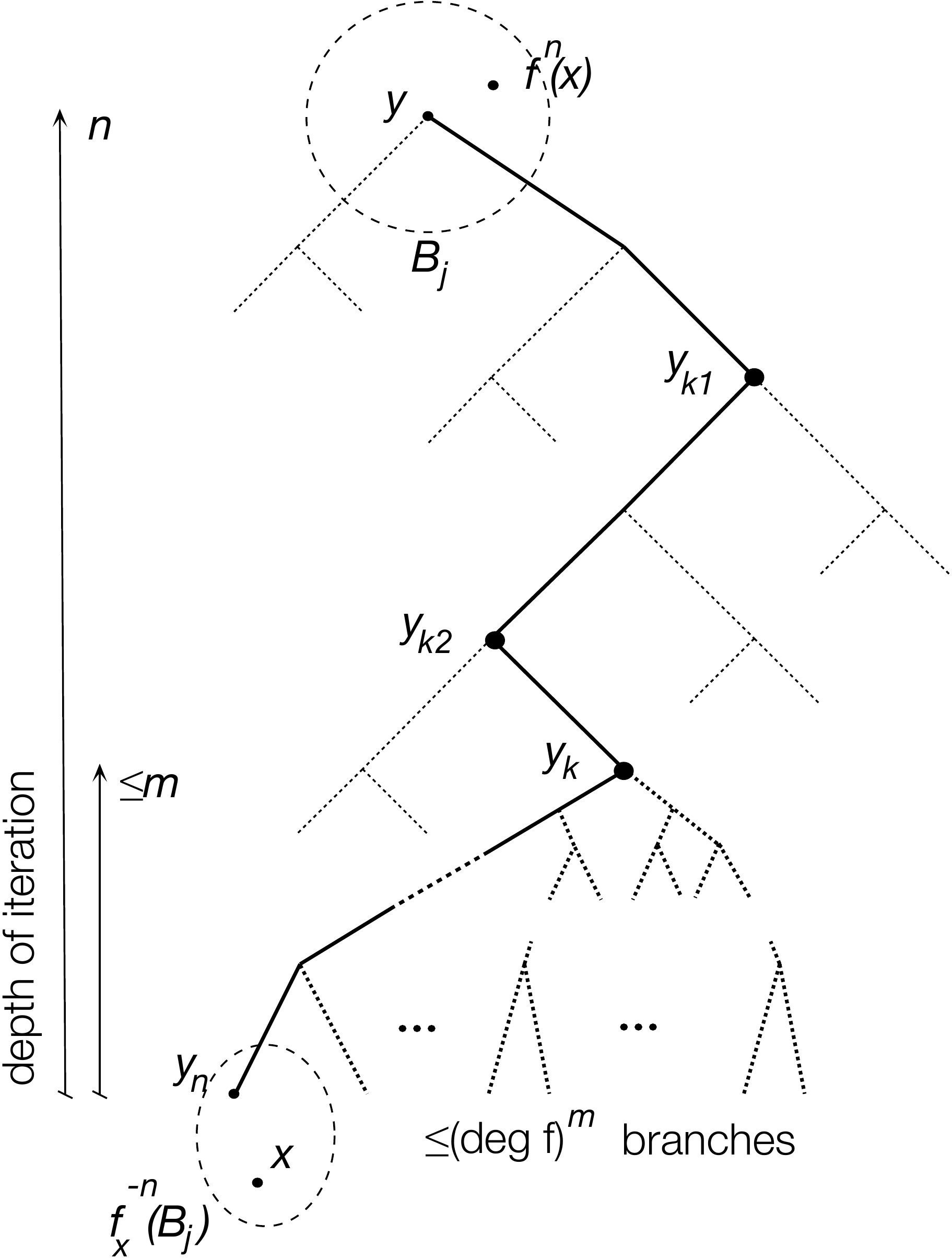}
    \caption{Pullback construction starting from the point $y$.}\label{Fig.reia}
 \end{figure}

We first note that Lemma \ref{telescope} and $\rho(f^n(x),y) \leq
r$ imply that
\[\rho(f^{n-k}(x),y_k) \leq r\,K_1.\]
This implies $B(f^{n-k}(x), 2r)\subset B(y_k,R)$ and hence $k\geq
n-m$ because $x\in G(m,\sigma)$. Thus, with fixed $n$ and $B_j$,
by Lemma~\ref{prs} the point $x$ must belong to one of at most
$K_2 e^{n\varepsilon} (\deg f)^m$ preimages of $B_j$. As
$B(y,3r)\subset B(x,4r)$ and $4r<R_1$, by Lemma \ref{telescope}
this pre-image of $B_j$ has diameter not greater than $8r\,K_1
e^{-n(\sigma - \varepsilon)}$.

Hence we showed that every point in $G(m,\sigma,n)$ belongs to
the $n$th pre-image of some ball $B_j$ along a backward branch for
that $k\ge m-n$ (where $k=\max_\ell k_\ell$ is as in the pullback
construction). Thus, by Lemma~\ref{telescope} the set $G(m,
\sigma, n)$ is contained in a union of at most $L K_2
e^{n\varepsilon} (\deg f)^m$ sets of diameter not greater than
$8r\,K_1 e^{-n(\sigma - \varepsilon)}$. Using those sets to cover
$G(m, \sigma, n)$ and applying~\eqref{eqn:prs}, we obtain
\[
\dim_{\rm H}G(m, \sigma) \leq \frac{ \varepsilon }{\sigma -
\varepsilon} .
\]
As $\varepsilon$ can be chosen arbitrarily small, the claim
follows.

Finally note that the set of points that we want to estimate in
the proposition is contained in the union
\[
\bigcup_{m=1}^\infty\bigcup_{n=1}^\infty G\big(m,\frac 1 n\big).
\]
Since for any set in this union its Hausdorff dimension is zero,
the assertion follows.
\end{proof}

Note that in the above proof we were able to show something more.
Namely notice that the choice of $r$ depended on $\sigma$ alone
(and not directly on the point $x$).

We are now prepared to prove the following estimate.

\begin{proposition}\label{prop:kjof}
Let $0<\alpha\leq \beta \leq \alpha^+$. We have
    \[
    \dim_{\rm H}\cL(\alpha,\beta) \le \max\left\{0,\max_{\alpha\le q\le \beta}
F(q)\right\}.
    \]
\end{proposition}

\begin{proof}
Let $\alpha$ and $\beta$ be like in the assumptions and let $x\in\cL(\alpha,\beta)$.
By Proposition~\ref{prop:wander} we can restrict our considerations to the case that $x$
is a conical point with corresponding number $r>0$, sequence $\{n_k\}_k$, and family of
neighborhoods $\{U_k\}_k$. Hence, there exist numbers $\delta>0$ and $K>1$ such
that
\begin{equation} \label{est2kj}
    2\delta\,\lvert(f^{n_k})'(x)\rvert^{-1} \, K^{-1}
    \le \diam f^{-n_k}_x(B(f^{n_k}(x),\delta))
    \le 2\delta \,\lvert(f^{n_k})'(x)\rvert^{-1} \, K .
\end{equation}
Choosing, if necessary, a subsequence of $\{n_k\}_k$, we can find
$q=q(x)\in [\alpha, \beta]$ for which we have
\begin{equation} \label{est2kjj}
\lim_{k\to\infty} \frac{1}{n_k}\log |(f^{n_k})'(x)| = q.
\end{equation}

Recall that for any $d\in\bR$ there exists a
$\exp\left(P(\varphi_d)-\varphi_d\right)$-conformal measure
$\nu_d$ that gives positive measure to any open set that
intersects $J$ (see Section~\ref{sec:conf}). Hence there exists a number 
$c_\delta>0$ for which for every $n_k$ we have $c_\delta\le
\nu_d(B(f^{n_k}(x),\delta))\le 1$. Using again  distortion
estimates, we can conclude that
\begin{multline}\label{huch2}
    c \,e^{-n_kP(\varphi_d)}K^{-1} \lvert (f^{n_k})'(x)\rvert^{-d} \\
\le \nu_d( f^{-n_k}_x(B(f^{n_k}(x),\delta)))\le
    e^{-n_kP(\varphi_d)}K \lvert (f^{n_k})'(x)\rvert^{-d}.
\end{multline}
The rest of the proof is similar to the proof of Proposition~\ref{prop:1}.
First we obtain that for every $d\in\bR$ we have
\[
\underline{d}_{\nu_d}(x) \le \frac{P(\varphi_d)}{q} +d.
\]
Choosing the right number $d$, we then conclude that
$\underline{d}_{\nu_d}(x) \le  F(q)$. We finish the proof
by constructing a measure such that for an arbitrarily small chosen number $\varepsilon$
and some number $q\in [\alpha,\beta]$ the lower pointwise dimension of that measure at
every $x\in
\cL(\alpha,\beta)$ is not greater than $F(q)+\varepsilon$.
\end{proof}

\begin{remark}
Notice that  in the general case, in which we do have critical points in $J$, for a point
$x\in J$ with $\underline{\chi}(x)< \overline{\chi}(x)$ we cannot apply the same
techniques as in Section~\ref{sec:nocrit}. In particular, in the above proof for each
point $x\in \cL(\alpha,\beta)$ we only know that there exists some number $q=q(x)\in
[\underline{\chi}(x), \overline{\chi}(x)]$ to which the Lyapunov exponents over some
subsequence of times $\{n_k\}_k$ will converge, while in Proposition~\ref{prop:1} we were
able to take an arbitrary number  $q$ in that interval. Hence, to show the following
result we can only consider the set $\cL(\alpha)=\cL(\alpha,\alpha)$ and not $\ccL(q,q)$
for an arbitrary number $q\in[\alpha,\beta]$, and the following corollary is weaker than
Corollary~\ref{cor:1}. However, the implications for the regular part of the spectrum
remain the same.
\end{remark}

Proposition~\ref{prop:kjof} and $\cL(\alpha)=\cL(\alpha,\alpha)$ readily imply the
following result.

\begin{corollary}
For $\alpha\in [\alpha^-, \alpha^+]\setminus\{0\}$ we have
\[
\dim_{\rm H} \cL(\alpha) \le F(\alpha).
\]
\end{corollary}

\section{Lower bounds for the dimension}\label{sec:4}

In this section we will derive lower bounds for the Hausdorff
dimension. We will either assume that $f$ is a non-exceptional map
or that $f$ is exceptional but $\Sigma\cap J=\emptyset$. Recall
that under those assumptions Proposition~\ref{decaf} is valid, so
we can approximate the pressure with respect to $f|_J$ with pressures that are
defined with respect to a sequence of Cantor repellers  $f^{a_m}|_{\Lambda_m}$ that were constructed in Section~\ref{sec:hypsubpress}.

One more case we would like to exclude is
$\alpha^-=\alpha^+=\alpha$. It is not very interesting because in
this case any measure supported on any hyperbolic set
$\Lambda_m\subset J$ has Lyapunov exponent $\alpha$. Hence we
automatically have
\[
\dim_{\rm H} \cL(\alpha) \geq \sup_{m\ge 1} F_m(\alpha)
\]
and the supremum on the right hand side is in this case equal to
$F(\alpha)$. Therefore, in the following considerations we will assume that
$\alpha^-<\alpha^+$.

\subsection{The interior of the spectrum}\label{sec:noche}

We use the  sequence of Cantor repellers to obtain for any exponent from the interior of the spectrum a big uniformly expanding subset of points with Lyapunov exponent $\alpha$ that provides us with an estimate from below. 

\begin{proposition} \label{jen}
    For $\alpha\in(\alpha^-,\alpha^+)$ we have $\dim_{\rm H}\cL(\alpha)\ge F(\alpha)$.
\end{proposition}

\begin{proof}
Let us consider the sequence of Cantor repellers $f^{a_m}|_{\Lambda_m}$ from Proposition~\ref{decaf}. By~\eqref{alla}, for each number
$\alpha^+>\alpha>\alpha^-$ there exists $m_0\ge 1$ such that
$\alpha_m^+>\alpha> \alpha_m^-$ for every $m\ge m_0$. Obviously we
have
\[
\dim_{\rm H}\cL(\alpha) \ge \sup_{m\ge 1}\dim_{\rm
H}\cL(\alpha)\cap\Lambda_m.
\]
Since $\Lambda_m$ is a uniformly expanding repeller with respect
to $f^{a_m}$, for any exponent $\alpha\in(\alpha_m^-,\alpha_m^+)$
there exists a unique number $q=q(\alpha)\in \bR$ such that
$\alpha = - \frac 1 {a_m}
\frac{d}{ds}P_{f^{a_m}|\Lambda_m}(\varphi_s)|_{s=q(\alpha)}$ and
an equilibrium state $\mu_q$ for the potential $\varphi_q$ (with
respect to $f^{a_m}|_{\Lambda_m}$) such that the Lyapunov exponent
of $\mu_q$ with respect to $f^{a_m}$ is equal to $a_m\alpha$
(compare the classical results in the introduction). For the
measure
\[
\nu_m=\sum_{i=0}^{a_m-1} \mu_m\circ f^i
\]
we have $\chi(\nu_m)=\alpha$. Hence, the variational principle
implies
\[\begin{split}
\max\Big\{ h_\nu(f)\colon &\nu\in\cM_{\rm E}\Big(\bigcup_i f^i(\Lambda_m)\Big),
\chi(\nu)=\alpha\Big\}\\
& \ge \frac 1 {a_m} P_{f^{a_m}|\Lambda_m}(S_{a_m}\varphi_q)+q\alpha \\
& \ge \inf_{d\in\bR} \left( \frac 1 {a_m}
P_{f^{a_m}|\Lambda_m}(S_{a_m}\varphi_d)+d\alpha \right) =
F_m(\alpha).
\end{split}\]
We obtain that $\dim_{\rm H}\nu=\frac{h_\nu(f)}{\chi(\nu)}$
whenever $\nu$ is an $f$-invariant ergodic Borel probability
measure with positive Lyapunov exponent~\cite{Man:88}. This
implies that for every $m\ge m_0$
\[
\dim_{\rm H}\cL(\alpha)\cap \bigcup_i f^i(\Lambda_m) \ge
\max\Big\{ \frac {h_\nu(f)} {\alpha}\colon \nu\in\cM_{\rm E}\Big(\bigcup_i
f^i(\Lambda_m)\Big), \chi(\nu)=\alpha\Big\}.
\]
From here the we can conclude that $\dim_{\rm H}\cL(\alpha)\ge
\sup_{m\ge 1}F_m(\alpha)$. Together with Proposition~\ref{decaf}
the statement is proved.
\end{proof}

\subsection{The boundary of the spectrum}\label{sec:clos}

Unfortunately, the above approach does not suffice to analyze the level sets for exponents from the boundary of the spectrum.
Our main goal in this section is to prove the following result. It will not only enable us to describe the boundary of the spectrum  but also provide us with dimension lower bounds for  level sets of irregular points.

\begin{theorem} \label{closing}
Let  $\{\Lambda_i\}_i$ be a sequence of subsets of $J$. We assume
that each $\Lambda_i$ is a uniformly expanding repeller for some
iteration $f^{a_i}$ and contains non-immediately postcritical points. 
Let $\{\phi_i\}_i$ be
a sequence of H\"older continuous potentials and let $\{\mu_i\}_i$
be  a sequence of equilibrium states for $\phi_i$ with respect to
$f^{a_i}|_{\Lambda_i}$.
    Then
    \[
    \dim_{\rm H}\left\{x\in J\colon
    \underline\chi(x)=\liminf_{i\to\infty}\chi(\mu_i),\,
    \overline\chi(x) = \limsup_{i\to\infty}\chi(\mu_i)\right\}
    \ge \liminf_{i\to\infty} \dim_{\rm H}\mu_i
    \]
    and
    \[
    \dim_{\rm P}\left\{x\in J\colon
    \underline\chi(x)=\liminf_{i\to\infty}\chi(\mu_i),\,
    \overline\chi(x) = \limsup_{i\to\infty}\chi(\mu_i)\right\}
    \ge \limsup_{i\to\infty} \dim_{\rm H}\mu_i.
    \]
\end{theorem}

We derive the following estimates for level sets that include exponents at the boundary
of the spectrum.

\begin{proposition}
    For $\alpha^-\leq \alpha<\beta\leq \alpha^+$ we have
    \[
    \dim_{\rm H}\widehat\cL(\alpha,\beta)
    \ge \max_{\alpha\le q\le\beta}F(q).
    \]
\end{proposition}

\begin{proof}
The claimed estimate follows from Theorem~\ref{closing}. 

We can also observe that for every $q\in (\alpha,\beta)$ we have
$\ccL(\alpha,\beta)\supset \cL(q,q)=\cL(q)$. Hence we can apply
Proposition~\ref{jen} to derive $\dim_{\rm H}\cL(q) \ge F(q)$ and prove the claimed estimate.
\end{proof}

\begin{proposition}\label{coco}
    For $\alpha^-\le \alpha\le\beta\le\alpha^+$ we have
    \[
    \dim_{\rm H}\cL(\alpha,\beta)
    \ge 
    \min\{F(\alpha),F(\beta)\}.
    \]
\end{proposition}

\begin{proof}
    Since we assume $\alpha^-<\alpha^+$, there exists a sequence $\{\Lambda_i\}_i$ of
uniformly expanding repellers and a sequence  $\{\mu_i\}_i$ of equilibrium states for the
potential $-\log\,\lvert f'\rvert$ with respect to $f|_{\Lambda_i}$ such that  $\dim_{\rm
H}\mu_i=F(\chi(\mu_i))$ and $ \alpha=\lim_{i\to\infty}\chi(\mu_{2i})$ and
$\beta=\lim_{i\to\infty}\chi(\mu_{2i+1})$. Theorem~\ref{closing} then implies that
    \[
    \cL(\alpha,\beta)
     = \big\{x\in J\colon \underline\chi(x) = \alpha,\overline\chi(x) = \beta \big\}
    \ge \liminf_{i\to\infty} F(\chi(\mu_i)).
    \]
Since $\liminf_{i\to\infty} F(\chi(\mu_i)) \ge\min\{ F(\alpha),F(\beta)\}$, this proves
the claimed estimate.
\end{proof}

Recall that a point $x$ is said to be \emph{recurrent} if $f^{n_i}(x)\to x$
for some sequence $n_i\nearrow\infty$.

\begin{corollary}
Assume that $J$ does not contain any recurrent critical points of $f$, that
$f$ is non-exceptional, and that $F(0)\neq -\infty$. Then we have
\[
\dim_{\rm H} \cL(0) = \dim_{\rm H} J\, = F(0).
\]
\end{corollary}

\begin{proof}
Proposition~\ref{coco} implies that $\dim_{\rm H} \cL(0) \geq
F(0)$. As $F(0)\ne-\infty$, the pressure function $d\mapsto
P(\varphi_d)$ is nonnegative and $F(0) = \inf\{d\colon
P(\varphi_d)=0\}$. Since $J$ does not contain any recurrent
critical points of $f$, we can apply~\cite[Theorem 4.5]{MUsurv}
and conclude that $\inf\{d\colon P(\varphi_d)=0\}=\dim_{\rm H}J$.
\end{proof}

To prove Theorem~\ref{closing}, we will construct a
sufficiently large set of points that have precisely the given
lower/upper Lyapunov exponent. Note that such a set will not ``be
seen" by any invariant measure in the non-trivial case that lower
and upper exponents do not coincide. Also points with Lyapunov
exponent zero, though Lyapunov regular, will not  ``be seen"  by
any interesting invariant measure.  Our approach is to show that
such a set is large with respect to some  necessarily
non-invariant probability measure. It is generalizing the
construction of so-called w-measures introduced in~\cite{GelRam:}
for which we strongly made use of the Markov structure that is
not available to us here.

\subsubsection*{A technical lemma}

In preparation for the proof of Theorem~\ref{closing} we start with some technical result that will be usefull shortly after.

\begin{lemma} \label{shadowing}
Let $g\colon \overline\bC\to \overline\bC$ be  a conformal map and $\Lambda\subset\overline\bC$ be a compact $g$-invariant hyperbolic topologically transitive set,  $\mu$ a
$g$-invariant ergodic measure on $\Lambda$ with Lyapunov exponent
$\chi=\chi(\mu)$, entropy $h=h_\mu(g)$, and Hausdorff dimension
$d=d(\mu)$. Let $V$ be an open set of positive measure $\mu$. If
$\gamma$ is small enough then for any $\varepsilon>0$ for
$\mu$-almost every point $v\in\Lambda$ there exist a number $K>0$ and a
sequence $\{n_i\}_i$ such that for each $n_i$ there is a set
$F_{n_i}\subset V\cap \Lambda$ such that for all $y_j\in F_{n_i}$
\begin{itemize}
    \item[i)] $g^{n_i}(y_j)=v$,
\vspace{0.1cm}
    \item[ii)] $K^{-1} \exp (m(\chi-\varepsilon)) <
        \vert(g^m)'(y_j)\rvert < K \exp (m(\chi+\varepsilon))$ for all $m\leq n_i$,
\vspace{0.1cm}
    \item[iii)] the branch $g^{-n_i}_{y_j}$ mapping $v$ onto $y_j$ extends to all
        $B(v,\gamma)$ and the distortion of the resulting map is bounded by $K$,
        \vspace{0.1cm}
    \item[iv)] we have
        \[
    \mu\left(\bigcup_{y_j\in F_{n_i}} g^{-n_i}_{y_j}(B(v,\gamma))\right)\ge K^{-1},
    \]
    \item[v)] for $j\ne k$ we have
    \[
    \rho\left(g^{-n_i}_{y_j}(B (v,\gamma)), g^{-n_i}_{y_k}(B (v,\gamma))\right)
     > K^{-1} \diam g^{-n_i}_{y_j}(B(v,\gamma)) ,
     \]
    \item[vi)] for any $x\in V$ and $r>0$ we have
        \[
        \mu\left(B(x,r)\,\cap\, \bigcup_{y_j\in F_{n_i}}
g^{-n_i}_{y_j}(B(v,\gamma))\right)
         \le Kr^{d-\varepsilon},
        \]
    \item[vii)] we have
        \[
    K^{-1}e^{-n(h+\varepsilon)}\le
    \mu\left(g^{-n_i}_{y_j}(B(v,\gamma))\right) \le
    e^{-n(h-\varepsilon)} .
    \]
\end{itemize}
\end{lemma}

\begin{proof}
As $g|_\Lambda$ is uniformly expanding, there exist numbers $\gamma>0$ and
$c_1>0$ such that for every $n\geq 1$, $y\in\Lambda$, $v=f^n(y)\in
\Lambda$, and a backward branch $g^{-n}_y$ mapping $v$ onto $y$
for every $x_1$, $x_2\in B(v,\gamma)$ we have
\begin{equation} \label{eqn:dist}
\frac{\lvert (g^{-n}_y)'(x_1)\rvert}{\lvert
(g^{-n}_y)'(x_2)\rvert}\le c_1,
\end{equation}
and in particular the mapping $g^{-n}_y$ extends to all of
$B(v,\gamma)$. Moreover, there is some constant $c_2>0$ not
depending on $n$ or $y$ such that
\[
\diam g^{-n}_y(B(v,\gamma)) \le c_2\gamma.
\]
We assume that $\gamma$ is so small that for any two points $x$,
$y$ with  distance from $\Lambda$ and mutual distance $<c_2 \gamma$, for any point $x'\in g^{-1}(x)$ there is at most one
point $y'\in g^{-1}(y)$ such that $\rho(x',y')\leq c_2 \gamma$. This is
true for any small enough number $\gamma$ because $\Lambda$ is in positive
distance from any critical point of $g$. Let $\widetilde V\subset
V$ be such that $B(\widetilde V, c_2 \gamma)\subset V$ and that
$\widetilde V$ is nonempty and of positive measure, which is
possible whenever $\gamma$ is small enough.

Notice that $B(v,\frac 1 {c_2+1}\gamma)$ has positive
$\mu$-measure for $\mu$-almost every $v$. Choose such point $v$
and let
\[
\widetilde{U}\eqdef B\left(v,\frac 1 {c_2+1}\gamma\right) ,\quad
U\eqdef B(v,\gamma).
\]
 Let
\[
\delta \eqdef \frac 1 2
\min\left\{\mu(\widetilde{U}), \mu(\widetilde V)\right\}.
\]

Let $\varepsilon>0$. There is a set of points $\Lambda'\subset
\Lambda$ with $\mu(\Lambda')>1-\delta$ (actually, it can
be chosen to have arbitrarily large measure) and a number $N>1$
such that: for every  $n\ge N$ and for every $x\in\Lambda'$  we
have
    \begin{itemize}
    \item [(C1)]
    $\displaystyle
    \left\lvert \frac{1}{n}\# \left\{k\in\{0,\ldots,n-1\}\colon
    g^k(x)\in \widetilde{U}\cap\Lambda \right\} - \mu(\widetilde{U}) \right\rvert\le
\varepsilon$,\\
    \vspace{0.1cm}
    \item     [(C2)]
    $\displaystyle
        \left\lvert \log \,\lvert (g^n)'(x)\rvert -n\chi\right\rvert \le n\varepsilon,
        $\vspace{0.2cm}
    \item [(C3)]
    $\displaystyle \left\lvert \log\, {\rm Jac}_\mu g^n(x)-nh\right\rvert\le
n\varepsilon$,
    \vspace{0.2cm}
    \item [(C4)]
 for $r\le\diam \Lambda$ we have
\[
\mu(B(x,r))  < c_3 r^{d-\varepsilon}.
\]
    \end{itemize}
Here (C1), (C2) simply follow from ergodicity,  (C3) is a
consequence of the Rokhlin formula and ergodicity, and (C4)
follows from the definition of the dimension $d(\mu)$.

We choose now a family $\{E_n\}_{n\geq N}$ of subsets of
$\Lambda'$ such that for every $n\ge N$ the set $E_n$ is a maximal
$(n,\frac {c_2} {c_2 +1}\gamma)$-separated subset of $\Lambda'$.
Given $n\ge N$, let $V_n\eqdef E_n\cap \widetilde{V}$. For each $n\ge
N$ let
\[
F_n\eqdef \left\{x\in V_n\colon g^{-n}_x(U)\subset V \right\}.
\]
For every $z_j\in F_n$ let $V_{n,j}\eqdef
g^{-n}_{z_j}(\widetilde{U})$. Obviously, all the sets $V_{n,j}$ 
are pairwise disjoint.

By (C1) the trajectory of every point from $\widetilde V \cap
\Lambda'$ visits $\widetilde{U}$ at some time $n\ge N$ (in fact,
at infinitely many times). Let $x$ be such a point and $n$ be such
a time, that is, $g^n(x)\in\widetilde U$. Because $E_n$ is
maximal, $x$ must be $(n,\frac {c_2}{c_2+1} \gamma)$-close to some
point $z_j\in E_n$. Hence we have
\[
g^n(z_j) \in B\left(g^n(x),\frac{c_2}{c_2+1} \gamma\right) \subset
U
\]
 and
 \[x\in V_{n,j}\subset
g^{-n}_{z_j}(U)\subset B(x, c_2 \gamma)\subset V.
\]
This shows that $z_j\in F_n$. This together with (C1) implies that
\[
\sum_{\ell=0}^{n-1}\mu\left(\bigcup_j V_{\ell,j}\right)
 \ge (\mu(\widetilde U)-\varepsilon) \, n\, \mu(\Lambda')
\]
and hence, as it is satisfied for all $n>N$, we obtain
\[
\sum_j\mu\left(V_{\ell,j}\right)
\ge \frac{1}{2}(\mu(\widetilde U)-\varepsilon)  \mu(\Lambda')
\]
for infinitely many $\ell\ge N$.
\begin{figure}[h]
         \includegraphics[width=5cm]{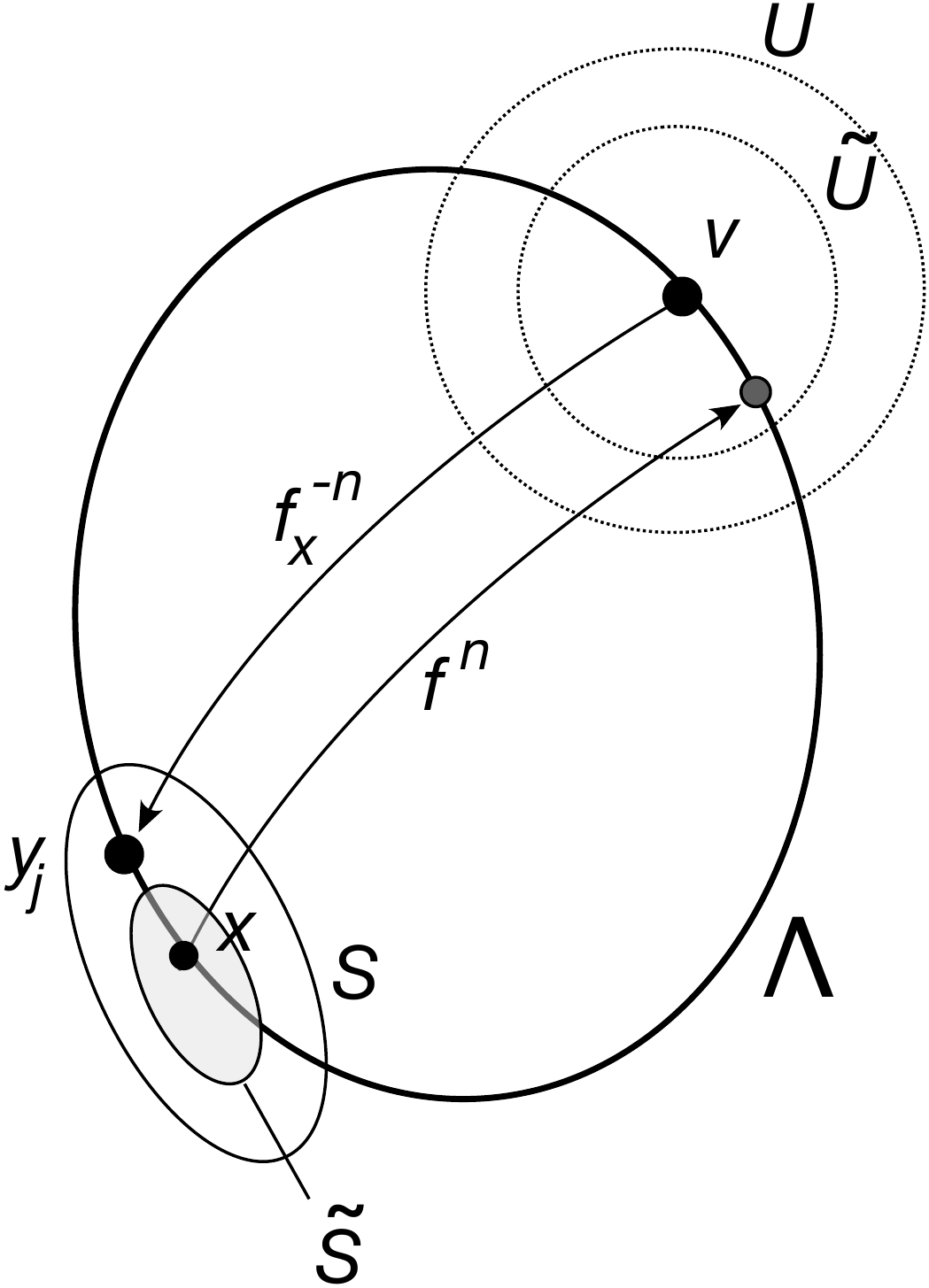}
    \caption{Finding $y_j$}\label{heute}
 \end{figure}
This gives us a sequence of times $n_i\eqdef \ell$ and points
$y_j=g^{-n_i}_{z_j}(v)$ for which the assertion of the lemma is
true. Indeed, $g^{-n_i}_{z_j}(U)$ and $g^{-n_i}_{z_k}(U)$ for
$j\neq k$ are disjoint because we are in a sufficiently large distance
from $\Crit$. We also know that those maps $g^{-n_i}_{z_j}|_U$
have uniformly bounded distortion. Recall that $U=B(v,\gamma)$,
which implies that $g^{-n_i}_{z_j}(B(v,\gamma/2))$ and
$g^{-n_i}_{z_k}(B(v,\gamma/2))$  for $j\neq k$ are not only
disjoint but in distance that is comparable to the sum of their
diameters and v) follows.
Further 
properties i) and iii) are checked directly from the construction. Moreover,
vi) follows from (C4), iv) follows from our choices of $\ell$, ii)
from (C2) and vii) from  iv) and (C3).
\end{proof}

\subsubsection*{Construction of  a Cantor set}

We now continue
with some preliminary constructions that will be needed in the proof of Theorem~\ref{closing}. 

As a first step,  we are going to construct ``bridges'' between the repellers $\Lambda_i$  (compare Figure~\ref{Fig.heute}). This is very similar to the proof of Lemma~\ref{bridge} though here we will not necessarily require that the repellers are disjoint. 
To fix some notation, let $\{(B_i,b_i)\}_i$ be a collection of \emph{bridges}, that
is, let
\[
B_i\eqdef B(z_i,r_i) \subset B(\Lambda_i,\gamma_i),
\]
where the numbers $r_i$, $b_i$, $\gamma_i$ and the points $z_i$ are appropriately chosen such
that $f^{b_i}|_{B_i}$ is a homeomorphism and that
\[
\mu_{i+1}(f^{b_i}(B_i))>0, \quad f^{b_i}(z_i)\in \Lambda_{i+1},
\]
and 
\[ \rho(f^k(B_i),\Crit) \geq \delta_i\quad \text{ for every } 1\le k\le b_i.
\]
The particular choice of the numbers and points will be specified in the following so that
we are able to apply Lemma~\ref{shadowing} to the each of the sets $\Lambda_i$ and the maps $g=f^{a_i}$.


Let us outline the following Cantor set construction.
Lemma~\ref{shadowing} enables us to select sufficiently many preimages for each of the disks $B_i$ (as  we choose  $B_i\subset B(v_i,\gamma_i)$ and then select preimages $g^{-n_i}(v_i)$ using the lemma).
The construction of this Cantor set is easily described in terms of backward branches: at level $i$ we start with the disk $B_i$, apply Lemma~\ref{shadowing} and  find a large number of components of $f^{-a_i n_i}(B_i)$ in $f^{b_{i-1}}(B_{i-1})$, then we go backwards through the bridge  obtaining the components in $B_{i-1}$. We repeat the procedure in $B_{i-1}$, $\ldots$, $B_1$.

Given the sequence of uniformly expanding repellers
$f^{a_i}|_{\Lambda_i}$ with equilibrium states $\mu_i$ with
Lyapunov exponent $\chi(\mu_i)$ and entropy $h_{\mu_i}(f^{a_i})$
(both with respect to the map $f^{a_i}$), let us denote
\[
\chi_i\eqdef \frac 1 {a_i} \chi(\mu_i), \quad h_i\eqdef
\frac 1 {a_i} h_{\mu_i}(f^{a_i}),\quad d_i\eqdef d(\mu_i) = \frac
{h_i} {\chi_i},
\]
and
\[
s_i \eqdef \lvert (f^{b_i})'(z_i)\rvert, \quad
t_i \eqdef \dist f^{b_i}|_{B_i}.
\]
Let
\[
w_i\eqdef \inf_{x\colon \rho(x,\Crit)>\delta_i}\lvert f'(x)\rvert,
\quad
W\eqdef\sup_{x\in J}\,\lvert f'(x)\rvert.
\]
The constants $\gamma_i$ can be chosen to be arbitrarily small
(at the cost of decreasing $r_i$ and increasing $b_i$, thus
changing $t_i$ and $s_i$ accordingly). In particular we can choose
each $\gamma_i$ sufficiently small such that Lemma~\ref{shadowing}
applies to  the map $g=f^{a_i}$ and the set $\Lambda_i$.
 Let
\[
V_1\eqdef B(\Lambda_1,\gamma_1)\quad\text{ and }\quad
V_{i+1}\eqdef f^{b_i}(B_i).
\]
We choose a fast decreasing sequence $\{\varepsilon_i\}_i$. We
will denote by $K_i$ the numbers $K_i=K(\Lambda_i, \mu_i,
V_i,\varepsilon_i,v_i)$ as given by Lemma~\ref{shadowing} (for
$g=f^{a_i}$) where we choose each point $v_i\in\Lambda_i$ such that
$B_i\subset B(v_i,\gamma_i)$. Let $A_1\eqdef f^{-a_1
n_1}_{y_j}(B_1)$ for one point $y_j$ and $n_1$ as provided by
Lemma~\ref{shadowing}. We have $\dist f^{a_1 n_1}|_{A_1}\le K_1$
and for every $x\in A_1$ and $m\leq n_1$ we have
\[
K_1^{-1} e^{a_1 m( \chi_1-\varepsilon_1) }< \lvert (f^{a_1
m})'(x)\rvert
 < K_1 e^{a_1 m(\chi_1+\varepsilon_1)}
\]
(note that $n_1$ can be chosen to be arbitrarily big, as guaranteed by
Lemma~\ref{shadowing}).

We apply Lemma~\ref{shadowing} now to $\Lambda_2$, $\mu_2$, $V_2$,
$\varepsilon_2$, $v_2$, $g=f^{a_2}$, which provides us with a
family $F_2\in f^{-a_2 n_2}(v_2)$. Let
$\widehat{A}_{2j}=f_{y_j}^{-a_2 n_2}(B(v_2,\gamma_2))$ for $y_j\in
F_2$. Those sets are contained in the set $V_2=f^{b_1}(B_1)$. We
also have $\dist f^{a_2 n_2}|_{\widehat A_{2j}}\le K_2$ and for
every $x\in \widehat A_{2j}$ and $m\leq n_2$ we have
\[
K_2^{-1} e^{a_2 m(\chi_2-\varepsilon_2)} < \lvert (f^{a_2
m})'(x)\rvert < K_2 e^{a_2 m(\chi_2+\varepsilon_2)}.
\]
Such sets will be used  later to distribute the w-measure
accordingly. It provides us also with a family of sets
$\widetilde{A}_{2j}\subset\widehat A_{2j}$  such that $f^{a_2
n_2}(\widetilde{A}_{2j})=B_2$. Such sets will be used later to
define our Cantor set on which the w-measure is supported.

\begin{figure}[h]
         \includegraphics[width=11cm]{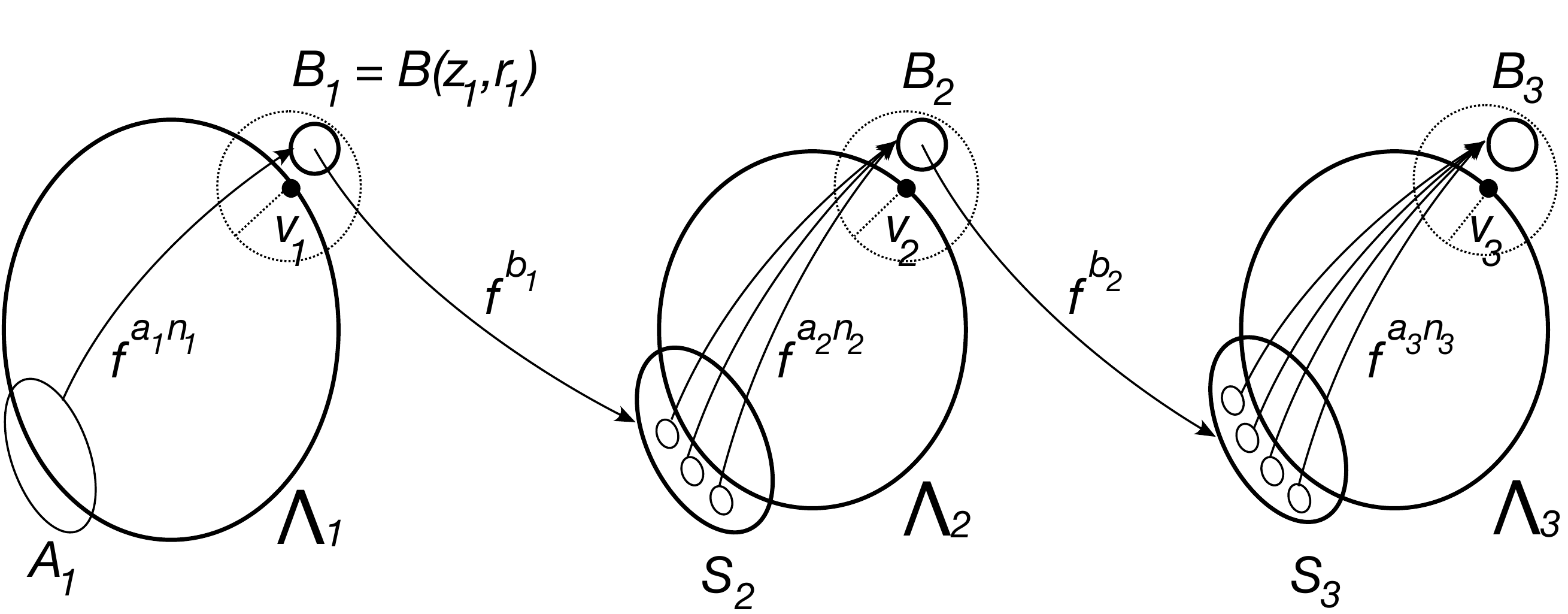}
    \caption{Connecting the hyperbolic sets by bridges.}\label{Fig.heute}
 \end{figure}
 
We repeat this procedure for every $i$, using
Lemma~\ref{shadowing} repeatedly: for every $i>1$  we find a
family $F_i\subset f^{-a_i n_i}(v_i)$ and the corresponding
components $\widehat A_{ij}$ of $f^{-a_i n_i}(B(v_i,\gamma_i))$
contained in $V_i$ that satisfy
\[
\rho(\widehat{A}_{i j}, \widehat{A}_{i k}) \geq K_i^{-1} \diam \widehat{A}_{i j}, \quad
j\ne k,
\]
and for every $x\in \widehat{A}_{ij}$ and $m\leq n_i$
\[
K_i^{-1} e^{a_i m (\chi_i - \varepsilon_i)} < \lvert (f^{a_i
m})'(x)\rvert < K_i e^{a_i m (\chi_i +\varepsilon_i)}
\]
and
\[
\dist f^{a_i n_i}|_{\widehat{A}_{ij}} \leq K_i.
\]
In addition, by Lemma~\ref{shadowing} iv) we have
\[
\mu_i\left(\bigcup_j \widehat A_{ij}\right)\ge K_i^{-1},
\]
and by Lemma~\ref{shadowing} vi) for any $x\in V_i$ and $r>0$ we have
\[
\mu_i\left( B(x,r)\cap \bigcup_j \widehat A_{ij}\right) \le K_i \,r^{d_i-\varepsilon_i}.
\]
Note here that for every $i\ge 1$ we have
\[
r_i\, s_i\, t_i^{-1} \leq \diam V_{i+1} \leq r_i \,s_i\, t_i .
\]

Let
\[
m_i \eqdef \sum_{k=1}^i \left(a_k n_k + b_k\right).
\]
Let $A_{2\,j}$ be the component of $A_1 \cap
f^{-m_1}(\widetilde{A}_{2j})$ for which $f^{a_1 n_1}(A_{2\,j})
\subset B_1$. Similarly, let $A_{i\,j_1 \ldots j_{i-1}}$ be the
component of $A_{(i-1)j_1 \ldots j_{i-2}} \cap
f^{-m_{i-1}}(\widetilde{A}_{i\,j_{i-1}})$ for which $f^{a_{i-1}
n_{i-1}
+m_{i-2}}(A_{i\,j_1 \ldots j_{i-1}}) \subset B_{i-1}$.\\
The sets
\[
A_i\eqdef \bigcup_{j_1\ldots j_{i-1}} A_{i\, j_1\ldots j_{i-1}}
\]
form a decreasing sequence of unions of topological balls.
Moreover, the pair $\left(A_{(i-1) j_1 \ldots j_{i-2}}, \{A_{i\,
j_1 \ldots j_{i-2}\,k}\}_k\right)$ is an image of $(V_i,
\{\widetilde{A}_{i \,k}\}_k)$ under a branch of the map
$f^{-m_{i-1}}$, the distortion of that branch is bounded by
\[
\widetilde{K}_{i-1}\eqdef \prod_{k=1}^{i-1} K_k t_k ,
\]
and the absolute value of its derivative is between
\begin{equation}\label{distr}
L_{i-1}\eqdef\prod_{k=1}^{i-1} s_k^{-1} t_k^{-1} e^{ - a_k n_k
(\chi_k + \varepsilon_k)} \quad\text{ and }\quad \widehat
L_{i-1}\eqdef \prod_{k=1}^{i-1} s_k^{-1} t_k e^{-a_k n_k (\chi_k -
\varepsilon_k)}
\end{equation}
($\widetilde{K}_i, L_i, \widehat L_i$ depend on $i$ only).
\begin{figure}[t]
         \includegraphics[width=11cm]{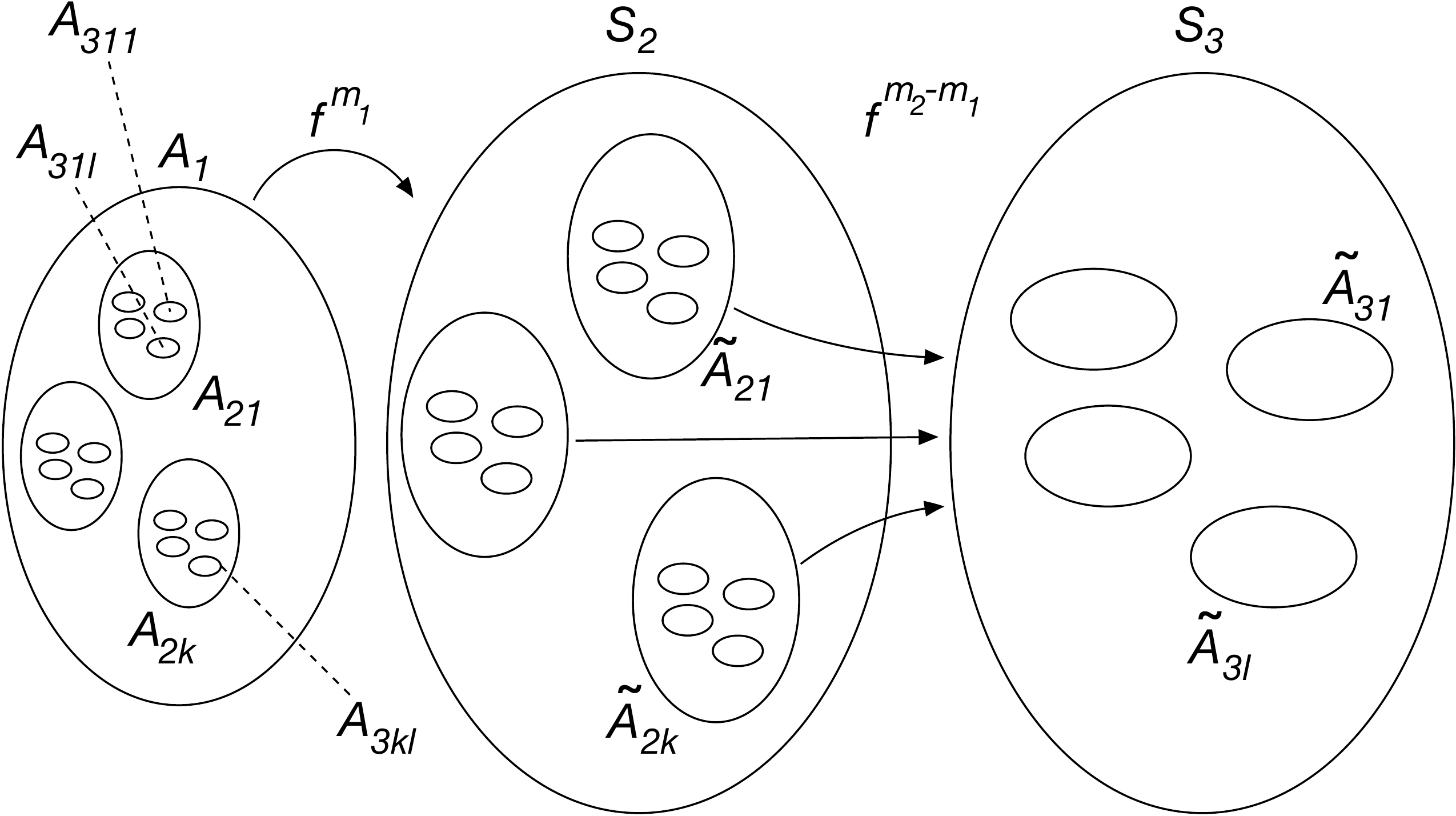}
    \caption{Local structure of the Cantor set $A$}\label{heu}
 \end{figure}
We can now define the Cantor set
\begin{equation}\label{cantorset}
A\eqdef\bigcap_{i\ge 1}A_i
\end{equation}
(compare Figure~\ref{heu}). We will summarize its geometric and
dynamic properties in the following lemma.

\begin{lemma} \label{aproposition}
The above defined set $A$ possesses the following properties:
\begin{itemize}
\item[i)] (Lyapunov exponents on the islands) for $x\in A_i$ and $k\leq n_i$ we have
\[
K_i^{-1} e^{a_i k (\chi_i - \varepsilon_i)} < \lvert
(f^{a_i k})'(f^{m_{i-1}}(x))\rvert < K_i e^{a_i k (\chi_i +
\varepsilon_i)},
\]
\item[ii)] (Lyapunov exponents on the bridges) for $x\in A_i$ and $k\leq b_i$ we have
\[
w_i^k < \lvert (f^k)'(f^{m_{i-1}+a_i n_i}(x))\rvert < W^k,
\]
\item[iii)] we have
\[
K_i^{-1} \widetilde{K}_{i-1}^{-1} L_{i-1} r_i \, e^{-a_i n_i
(\chi_i + \varepsilon_i)} \leq \diam A_{i j_1 \ldots j_{i-1}} \leq
K_i \widetilde{K}_{i-1} \widehat L_{i-1} r_i \,e^{- a_i n_i(\chi_i
- \varepsilon_i)},
\]
\item[iv)] there are at least $\exp(a_i n_i (h_i -
\varepsilon_i))$ sets $A_{i\, j_1 \ldots j_{i-1}}$ contained in
every set $A_{(i-1) j_1 \ldots j_{i-2}}$,\vspace{0.2cm} \item[v)]
for any $k_{i-1} \neq j_{i-1}$ and any $i, j_1, \ldots, j_{i-2}$
we have
\[
\rho(A_{i j_1 \ldots j_{i-1}}, A_{i j_1 \ldots j_{i-2} k_{i-1}}) \geq
\widetilde{K}_{i-1}^{-1} K_i^{-1} \diam A_{i j_1 \ldots j_{i-1}}
\]
\item[vi)] we have
\[
\mu_i\left(\bigcup_j\widehat A_{ij}\right) \ge K_i^{-1},
\]
\item[vii)] for any $x\in V_i$ and $r>0$ we have
\[
\mu_i\left(B(x,r)\cap\bigcup_j\widehat A_{ij}\right)\le K_i\,r^{d_i-\varepsilon_i},
\]
\item[viii)] we have
\[
    K_i^{-1}e^{-a_i n_i(h_i+\varepsilon_i)}\le
    \mu_i\left(\widehat A_{ij}\right) \le
    e^{-a_i n_i(h_i-\varepsilon_i)} .
\]
\end{itemize}
\end{lemma}


Additional assumptions on $\{n_i\}_i$ and  properties guaranteed by
Lemma~\ref{shadowing} will enable us to estimate the Hausdorff and packing dimensions of the
constructed set and describe the upper and lower Lyapunov exponents at each point.  This will be done in the following.

\subsubsection*{Construction of a w-measure on the Cantor set}

We continue to consider the Cantor set $A$ constructed in~\eqref{cantorset}.
Let $\mu$ be the probability measure which on each level
$i$ is distributed on the cylinder sets $A_{i\, j_1 \ldots
j_{i-1}}$ of level $i$ in the following way
\begin{equation}\label{wmeasu}
\mu(A_{i\, j_1 \ldots j_{i-1}}) \eqdef
\mu(A_{i-1\, j_1 \ldots j_{i-2}})
\frac{\mu_i(\widehat A_{i\,j_{i-1}})}
    {\sum_k\mu_i(\widehat A_{i\,k})}\,.
\end{equation}
We extend the measure $\mu$ arbitrarily to the Borel $\sigma$-algebra of $A$. We call the probability measure a \emph{w-measure} with respect to the sequence $\{f|_{\Lambda_i},\phi_i,\mu_i\}_i$.
\\[-0.2cm]

After these preparations, we are now able to prove Theorem~\ref{closing}. We will continue to use the notations in the above construction of the set $A$.

\begin{proof}[Proof of  Theorem~\ref{closing}]
In the course of the following proof we will choose some sequence $\{\varepsilon_k\}_k$ and then
construct a sequence of positive integers $\{n_i\}_i$. Here each of those numbers $n_i$ has to satisfy several conditions that depend on $\{\varepsilon_k\}_k$, the parameters of the
hyperbolic sets $\{\Lambda_k\}_k$ and the measures $\{\mu_k\}_k$, and the previously chosen
numbers $n_j$, $j=1$, $2$, $\ldots$, $i-1$. Naturally, it is always possible to satisfy all those conditions at the same time.

We will first check that the Cantor set defined in~\eqref{cantorset} satisfies
\[
A\subset \cL\Big(\liminf \chi(\mu_i), \limsup \chi(\mu_i)\Big)
\]
(under some appropriate assumptions about $\{n_i\}_i$). Then we
will estimate the Hausdorff and packing dimensions of $A$ using the w-measure $\mu$ defined in~\eqref{wmeasu}.

Let us first consider the Lyapunov exponent at a point in the set $A$.
Let
\[
\ell_n(x) \eqdef \frac 1 n \log\, \lvert(f^n)'(x)\rvert.
\]
For $n\leq a_1 n_1$ we have \[\ell_n(x) \in \left(\chi_1 -
\varepsilon_1 - \frac{1}{n}\log K_1 - O(\frac {a_1} n), \chi_1
+\varepsilon_1 +\frac{1}{n} \log K_1 + O(\frac {a_1} n)\right).\]
We know that $f^{n_1+k}(x)$ stays in distance at least $\delta_1$
from the critical points for every $k\leq b_1$. Hence, for $a_1
n_1<n \leq a_1 n_1 + b_1$ we have
\begin{equation}\label{ineq1}
\ell_n(x) = \frac {a_1 n_1} n \ell_{a_1 n_1}(x) + \frac {n- a_1 n_1} n
O(\lvert\log w_1\rvert)
\end{equation}
and for $n_1$ big enough (in comparison with $\log K_1$ and $b_1
\lvert\log w_1\rvert$) the right hand side of~\eqref{ineq1}  is
between $\chi_1 - 2\varepsilon_1$ and $\chi_1+2\varepsilon_1$.
For $m_{i-1} < n \leq m_{i-1}+a_i n_i$ we have
\[
\ell_n(x) = \frac {m_{i-1}} n \ell_{m_{i-1}}(x) + \frac 1 n
\log\,\lvert(f^{n-m_{i-1}})'(f^{m_{i-1}}(x))\rvert + O(\frac {a_i}
n)
\]
while for $m_{i-1} + a_i n_i < n \leq m_i$ we have
\[
\ell_n(x) = \frac {m_{i-1}} n \ell_{m_{i-1}}(x) + \frac 1 n \log\,
\lvert(f^{a_i n_i})'(f^{m_{i-1}}(x))\rvert + \frac {n-m_{i-1} -
a_i n_i} n O(\lvert\log w_i\rvert).
\]
Estimating the second summand using Lemma~\ref{aproposition} i)
and the third one using Lemma~\ref{aproposition} ii) and assuming
that $n_i$ is big enough (in comparison with $n_{i-1}$, $b_i
\lvert\log w_i\rvert$, $a_{i+1}$, and $\log K_i$), we can first
prove that
\[
\left\lvert \ell_{m_i}(x) - \chi_i \right\rvert < 2 \varepsilon_i
\]
(by induction) and then prove that for all $m_i < n < m_{i+1}$ we have
\[
\left\lvert \ell_n(x) - \left(\frac {m_i} n \chi_i + \frac {n-m_i} n
\chi_{i+1}\right)\right\rvert
< 2(\varepsilon_i + \varepsilon_{i+1}).
\]
As the upper (the lower) Lyapunov exponent of $x$ equals the upper (the lower)
limit of $\ell_n(x)$ as $n\to\infty$, we have shown that every point $x\in A$
satisfies
\[
\underline\chi(x)=\liminf_{i\to\infty}\chi(\mu_i), \quad
\overline\chi(x)=\limsup_{i\to\infty}\chi(\mu_i).
\]

Let us now estimate the Hausdorff and packing dimensions of the set $A$.
To do so, we will
apply the Frostman lemma. 
Let us calculate the pointwise dimension of the measure $\mu$ defined in~\eqref{wmeasu} at an arbitrary point $x\in A$. Notice that we can write
\[
\{x\} = \bigcap_{i=1}^\infty A_{i\, j_1 \ldots j_{i-1}}
\]
for some appropriate symbolic sequence $(j_1j_2\ldots )$. By Lemma~\ref{aproposition} v), the ball  $B(x,r)$ does not intersect any of the sets $A_{i\, k_1 \ldots
k_{i-1}}$ if $k_1 \ldots k_{i-1} \ne j_1 \ldots
j_{i-1}$
 whenever we have
\[
r\le R_i(x)\eqdef  \widetilde{K}_{i-1}^{-1} K_i^{-1} \diam A_{i\, j_1 \ldots j_{i-1}}.
\]
Consider $R_{i+1}(x)<r\le R_i(x)$. We have
\[\begin{split}
\mu(B(x,r))
&\le \sum_{A_{i+1\,j_1\ldots j_{i-1}k}\cap B(x,r)\ne\emptyset}
    \mu(A_{i+1\,j_1\ldots j_{i-1}k})\\
&\le \sum_{A_{i+1\,j_1\ldots j_{i-1}k}
    \subset B\left(x,r+\max_\ell\diam A_{i+1\,j_1\ldots j_{i-1}\ell}\right)}
    \mu(A_{i+1\,j_1\ldots j_{i-1}k}).
\end{split}\]
Let $D_i\eqdef \max_\ell\diam A_{i+1\,j_1\ldots j_{i-1}\ell}$. We continue with
\[\begin{split}
\mu(B(x,r))
&\le \frac{ \mu(A_{i\,j_1\ldots j_{i-1}})}{\sum_\ell\mu_{i+1}(\widehat A_{i+1\,\ell})}
    \sum_{k\colon \widehat A_{i+1\,k}\subset f^{m_i}(B(x ,r+D_i )) }
    \mu_{i+1}(\widehat A_{i+1\,k})      \\
&\le    \frac{ \mu(A_{i\,j_1\ldots j_{i-1}})}{\sum_\ell\mu_{i+1}(\widehat A_{i+1\,\ell})}
    \mu_{i+1}\left( B\left(f^{m_i}(x),L_i^{-1}(r+D_i) \right)\right) .
\end{split}\]
Using Lemma~\ref{aproposition} vi) and vii) we can estimate
\begin{eqnarray}
\mu(B(x,r))
&\le&   \mu(A_{i\,j_1\ldots j_{i-1}}) K_{i+1}\,
    \mu_{i+1}\left( B\left(f^{m_i}(x),L_i^{-1}(r+D_i) \right)\right)\notag\\
&\le&       \mu(A_{i\,j_1\ldots j_{i-1}})
        \frac{K_{i+1}^2}{L_i^{d_{i+1}-\varepsilon_{i+1}}}
        \left(r+D_i\right)^{d_{i+1}-\varepsilon_{i+1}} . \label{poj}
\end{eqnarray}
Let
\begin{equation}\label{Xi}
    \Xi\eqdef   \mu(A_{i\,j_1\ldots j_{i-1}})
        \frac{K_{i+1}^2}{L_i^{d_{i+1}-\varepsilon_{i+1}}}.
        \end{equation}
Using Lemma \ref{aproposition} vi) and viii), we can now estimate the first factor in~\eqref{Xi} and  we obtain
\[\begin{split}
\log\mu(A_{i\,j_1\ldots j_{i-1}}) &\le
\sum_{k=2}^i \log \frac{\mu(A_{k\,j_1\ldots j_{k-1}})}{\mu(A_{k-1\,j_1\ldots j_{k-2}})}\\
&\le \sum_{k=2}^i \left( \log\mu_k\left(\widehat A_{k\,j_{k-1}}\right)+\log K_k\right)\\
&\le - \sum_{k=2}^i \left(a_k n_k(h_k-\varepsilon_k) +\log
K_k\right) \le -a_i n_i(h_i-2\varepsilon_i)
\end{split}\]
provided that we assume that $n_i$ has been chosen big enough (in comparison with
$a_k$, $n_k$, $\log K_k$, $k<i$). For the second factor in~\eqref{Xi} we yield
\[
-\log L_i = \sum_{k=1}^i\log s_k+\log t_k+a_k
n_k(\chi_k+\varepsilon_k) \le a_i n_i(\chi_i+2\varepsilon_i).
\]
provided that we assume that $n_i$ has been chosen big enough (in comparison with 
$a_k$, $n_k$, $\log K_k$, and $s_k$, $k<i$). This implies that
\[
\log\Xi = a_i n_i\chi_i \left(d_{i+1}-d_i\right) + n_i\,
O(\varepsilon_i,\varepsilon_{i+1}).
\]
Further,  using~\eqref{distr} and Lemma~\ref{aproposition} iii) we obtain
\[
\frac{\diam A_{i+1\,j_1\ldots j_{i-1}\ell}}{R_{i+1}}
\le \widetilde K_i^3\, K_{i+1}^3 \prod_{k=1}^{i+1}e^{2n_k\varepsilon_k}
\]
and hence
\[
D_i\le e^{3n_{i+1}\varepsilon_{i+1}}\, R_{i+1}.
\]
In addition, by Lemma~\ref{aproposition} iii) we have
\[
\log R_{i+1} = a_{i+1} n_{i+1}(\chi_{i+1}+O(\varepsilon_{i+1}))
\]
and
\[
\log D_i\le -a_{i+1} n_{i+1}(\chi_{i+1}+O(\varepsilon_{i+1})).
\]
assuming that $n_i$ has been chosen big enough. To
estimate~\eqref{poj}, we consider now the two cases: a) that $r\le
D_i$ and b) that $r> D_i$. In case a) we have that
\[
\mu(B(x,r))\le \Xi\,(2D_i)^{d_{i+1}-\varepsilon_{i+1}}.
\]
Note that the  the right hand side of this estimate no longer depends on $r$ and hence,
\begin{equation} \label{est11}
\frac{\log\mu(B(x,r))}{\log r}
\ge \frac{\log\left(\Xi\cdot(2D_i)^{d_{i+1}-\varepsilon_{i+1}}\right)}{\log R_{i+1}}
= O\left(\frac{n_i}{n_{i+1}}\right) + d_{i+1} + O(\varepsilon_{i+1}).
\end{equation}
In case b) we have
\[
\mu(B(x,r))\le \Xi\,(2r)^{d_{i+1}-\varepsilon_{i+1}}
\]
and hence
\[\begin{split}
\frac{\log\mu(B(x,r))}{\log r}
&\ge \frac{\log\left(\Xi\cdot(2\,r)^{d_{i+1}-\varepsilon_{i+1}}\right)}{\log r}.
\end{split}\]
We again need to distinguish two cases: If $d_{i+1}>d_i$, then
\begin{equation} \label{est12}
\frac{\log\mu(B(x,r))}{\log r} \ge \frac{\log\mu(B(x,R_i))}{\log R_i} \ge
d_i + O(\varepsilon_i,\varepsilon_{i+1}),
\end{equation}
while in the case $d_i\ge d_{i+1}$ we have
\begin{equation} \label{est13}
\frac{\log\mu(B(x,r))}{\log r} \ge \frac{\log\mu(B(x,R_{i+1}))}{\log R_{i+1}} \ge
d_{i+1} + O\left(\frac{n_i}{n_{i+1}}\right) + O(\varepsilon_i,\varepsilon_{i+1}).
\end{equation}
The estimations \eqref{est11},~\eqref{est12}, and \eqref{est13}
prove that if the sequence $\{n_i\}_i$ increases fast enough then for every $x\in A$ we
have
\[
\underline d_\mu(x)\ge \liminf_{i\to\infty} d_i
\quad
\text{ and }
\quad \overline d_\mu(x)\ge \limsup_{i\to\infty} d_i .
\]
Hence, applying the Frostman lemma, we obtain 
\[
\dim_{\rm H}A \ge  \liminf_{i\to\infty} d_i
\quad
\text{ and }
\quad
\dim_{\rm P}A \ge  \limsup_{i\to\infty} d_i ,
\] 
and hence the assertion of Theorem~\ref{closing} follows.
 \end{proof}
\bibliographystyle{amsplain}

\end{document}